\documentclass[11pt,a4paper,final]{amsart}
\usepackage{geometry}
\usepackage{filecontents}
\usepackage{hyperref}
\hypersetup{
    bookmarks=true,         
    unicode=false,          
    pdftoolbar=true,        
    pdfmenubar=true,        
    pdffitwindow=false,     
    pdfstartview={FitH},    
    pdftitle={My title},    
    pdfauthor={Author},     
    pdfsubject={Subject},   
    pdfcreator={Creator},   
    pdfproducer={Producer}, 
    pdfkeywords={keyword1, key2, key3}, 
    pdfnewwindow=true,      
    colorlinks=true,       
    linkcolor=blue,          
    citecolor=blue,        
    filecolor=magenta,      
    urlcolor=cyan           
}
\usepackage{enumerate} 
\usepackage{tikz}
\usepackage{stmaryrd} 
\usepackage{mathrsfs}
\usepackage{graphicx}

\usepackage{graphicx}
\usepackage{latexsym}
\usepackage{amssymb}  
\usepackage{amsmath, mathtools} 
\usepackage{dsfont}
\usepackage{chngcntr}
\usepackage{esint}

\usetikzlibrary{calc}
\usepackage[font=scriptsize]{caption}

\usepackage{enumerate}

\numberwithin{equation}{section}

\newtheoremstyle{thmlemcorr}{10pt}{10pt}{\itshape}{}{\bfseries}{.}{10pt}{{\thmname{#1}\thmnumber{ #2}\thmnote{ (#3)}}}
\newtheoremstyle{thmlemcorr*}{10pt}{10pt}{\itshape}{}{\bfseries}{.}\newline{{\thmname{#1}\thmnumber{ #2}\thmnote{ (#3)}}}
\newtheoremstyle{defi}{10pt}{10pt}{\itshape}{}{\bfseries}{.}{10pt}{{\thmname{#1}\thmnumber{ #2}\thmnote{ (#3)}}}
\newtheoremstyle{remexample}{10pt}{10pt}{}{}{\bfseries}{.}{10pt}{{\thmname{#1}\thmnumber{ #2}\thmnote{ (#3)}}}
\newtheoremstyle{ass}{10pt}{10pt}{}{}{\bfseries}{.}{10pt}{{\thmname{#1}\thmnumber{ A#2}\thmnote{ (#3)}}}

\theoremstyle{thmlemcorr}
\newtheorem{theorem}{Theorem}
\numberwithin{theorem}{section}
\newtheorem{lemma}[theorem]{Lemma}
\newtheorem{corollary}[theorem]{Corollary}
\newtheorem{proposition}[theorem]{Proposition}

\theoremstyle{thmlemcorr*}
\newtheorem{theorem*}{Theorem}
\newtheorem{lemma*}[theorem]{Lemma}
\newtheorem{corollary*}[theorem]{Corollary}
\newtheorem{proposition*}[theorem]{Proposition}
\newtheorem{problem*}[theorem]{Problem}
\newtheorem{conjecture*}[theorem]{Conjecture}

\theoremstyle{defi}
\newtheorem{definition}[theorem]{Definition}

\theoremstyle{remexample}
\newtheorem{remark}[theorem]{Remark}

\newenvironment{claim}[1]{\par\noindent\underline{Claim:}\space#1}{}

\theoremstyle{ass}

\newcommand{\Crm}{\mathrm{C}}

\newcommand{\Lrm}{\mathrm{L}}

\newcommand{\Wrm}{\mathrm{W}}

\DeclareMathOperator{\curl}{curl}
\DeclareMathOperator{\dist}{dist}

\DeclareMathOperator{\supp}{supp}

\newcommand{\setBB}[2]{\biggl\{\, #1 \ \ \textup{\textbf{:}}\ \ #2 \,\biggr\}}

\newcommand{\dd}{\;\mathrm{d}}

\newcommand{\R}{\mathbb{R}}

\newcommand{\toweakstar}{\overset{*}\rightharpoonup}

\newcommand{\BV}{\mathrm{BV}}

\def\XXint#1#2#3{{\setbox0=\hbox{$#1{#2#3}{\int}$} 
\vcenter{\hbox{$#2#3$}}\kern-.5\wd0}}

\renewcommand{\phi}{\varphi}
\newcommand{\M}{\mathcal M}

\newcommand{\wc}{\rightharpoonup}

\newcommand{\dx}{\,\mathrm{d}x}

\DeclareMathOperator{\e}{\varepsilon}
\DeclareMathOperator{\hex}{H_{ex}}

\DeclareMathOperator{\Per}{Per}

\DeclareMathOperator{\A}{\mathscr A}
\DeclareMathOperator{\B}{\mathscr B}
\DeclareMathOperator{\Dev}{Dev}

\DeclareMathOperator{\Id}{id}

\renewcommand{\ker}{\textnormal{Ker}}

\title[Regularity of free-interface variational problems]{Regularity for free interface variational problems in a general class of gradients} 

\author[A.~Arroyo-Rabasa]{Adolfo Arroyo-Rabasa}
\address{\textit{Adolfo Arroyo-Rabasa:} Institut f\"ur Angewandte Mathematik, Universit\"at Bonn, 53115 Bonn, Germany.}
\email{adolfo.arroyo.rabasa@hcm.uni-bonn.de}

\begin{document}

\maketitle

\begin{abstract}
We present a way to study a wide class of optimal design problems with a perimeter penalization. 
More precisely, we address existence and regularity properties
of saddle points of energies of the form
\[
 (u,A) \quad \mapsto \quad \int_\Omega 2fu \dd x \; - \int_{\Omega \cap A}
 \sigma_1\A u\cdot \A u \, \dd x \; - \int_{\Omega \setminus A} \sigma_2\A u\cdot \A u \, \dd x \; + \; \text{Per}(A;\overline \Omega),
\]
where $\Omega$ is a  bounded Lipschitz domain, $A\subset \R^N$ is a Borel set, $u:\Omega \subset \R^N \to \R^d$, $\A$ is an operator of gradient form, and
$\sigma_1, \sigma_2$ are two not necessarily well-ordered symmetric tensors.
The class of operators of gradient form includes scalar- and vector-valued gradients, 
symmetrized gradients, and higher order gradients. Therefore, our results may be applied to a wide range 
of problems in elasticity, conductivity or plasticity models. 

In this context and under mild assumptions on $f$, we show for a solution $(w,A)$, that
the topological boundary of $A \cap \Omega$ is locally a $\Crm^1$-hypersurface up to a closed set of zero $\mathcal H^{N-1}$-measure. 
\vspace{4pt}


\vspace{4pt}

\vspace{4pt}

\noindent\textsc{Date:} \today{}.
\end{abstract}

\section{Introduction}
\label{sec1:introduction}
The problem of finding optimal designs involving two materials goes back to the work of Hashin and Shtrikman. In \cite{Hash63}, the authors made the first successful attempt to derive the optimal bounds of a composite material. It was later on, in the series of papers \cite{KSI,KSII,KSIII}, that Kohn and Strang 
described the connection between composite materials, 
the method of relaxation, and the homogenization theory developed by Murat and Tartar \cite{MurTar85,MurTar85ii}. 
In the context of homogenization, {\it better} designs tend to develop finer and finer geometries; a process which results in the creation of
non-classical designs. One way to avoid the mathematical abstract of infinitely fine mixtures is to add a cost on the interfacial energy. 
In this regard, there is a large amount of optimal design problems that involve an interfacial energy and a Dirichlet energy. The study of regularity properties in this setting has been mostly devoted to problems where the Dirichlet energy is related to a scalar elliptic equation; see \cite{AmbrosioButtazzo93,Lin93,LinKohn99,Larsen13,FuscoJulin15,DPFig15}, where partial $\Crm^1$-regularity on the interface is shown for an optimization problem oriented to find dielectric materials of maximal conductivity.
We shall study regularity properties of similar problems in a rather general framework. Our results extend the aforementioned results to linear elasticity and linear plate theory models. 

Before turning to a precise mathematical statement of the problem let us first present the model in linear plate theory that motivated our results.
Let $\Omega = \omega \times [-h,h]$ be the reference configuration of a plate of thickness $2h$ and
cross section $\omega \subset \R^2$. The linear equations governing a clamped plate $\Omega$ as 
$h$ tends to zero for the Kirchhoff model are
\begin{equation}\label{eqn:plate}
 \begin{cases}
  \nabla \cdot \nabla \cdot \big(\sigma \nabla^2 u\big) = &  f \quad \quad \text{in} \; \omega,\\
  \hfill \partial_\nu u = u = & 0 \hfill \quad \text{in} \; \partial \omega,
 \end{cases}
\end{equation}
where $u : \omega \to \R$ represents the displacement of the plate with respect to a vertical load
$f \in \Lrm^\infty(\omega)$, and the design of the plate is described by a symmetric positive definite
fourth-order tensor $\sigma$ (up to a cubic dependence on the constant $h$). Here, we denote the second gradient by
\[
 \nabla^2 u \coloneqq \left(\frac{\partial^2 u}{\partial x_i \partial x_j}\right)_{ij}, \quad i,j = 1,2.
\]
Consider the physical problem of a thin plate $\Omega$ made-up of two elastic
materials. More precisely,
for a given set $A \subset \omega \subset \R^2$ we define the symmetric positive tensor 
\[\sigma_A(x) \coloneqq  \mathds {1}_A \sigma_1 + (1 -  \mathds 1_A)\sigma_2,\]
where 
$\sigma_1, \sigma_2 \in 
\text{Sym}(\R^{2 \times 2},\R^{2 \times 2})$.
In this way, to each Borel subset $A \subset \omega$, there corresponds a displacement $u_A :\omega \to \R$ solving equation (\ref{eqn:plate}) with 
$\sigma = \sigma_A$. One measure of the rigidity of the plate is the so-called compliance, i.e., the work done by the loading. The smaller the
compliance, the stiffer the plate is. A reasonable optimal design model consists in finding the most rigid design $A$ under the aforementioned costs. One seeks to minimize
an energy of the form
\[
A \mapsto \int_\omega \sigma_A \nabla^2 u_A \cdot \nabla^2 u_A \, \dd x~ + ~\Per(A;\omega), \qquad \text{among Borel subsets $A$ of}\; \R^2.
\]
Optimality conditions for a stiffest plate can be derived by taking local variations on the design. 
For such analysis to be meaningful, one has to ensure first that the variational equations of optimality have a suitable meaning {\it in} the interface. 
Hence, it is natural to ask for the maximal possible regularity of $\partial A$ and
$\nabla^2 u_A$.

We will introduce a more general setting where one can replace the second gradient $\nabla^2$ by an operator $\A$ of {\it gradient type} (see
Definition \ref{gto} and the subsequent examples in the next section for a precise description of this class).

\subsection{Statement of the problem} Let $N \ge 2$, and let $d,k$ be positive integers. We shall work in $\Omega \subset \R^N$; a nonempty, open, and bounded Lipschitz domain.
We also fix a function $f \in \Lrm^\infty(\Omega;\R^d)$ and let $\sigma_1$ and $\sigma_2$ be two positive definite tensors in $\text{Sym}(\R^{dN^k} \otimes \R^{dN^k})$ satisfying a strong pointwise G{\aa}rding inequality: there exists a positive constant $M$ such that
\begin{equation}\label{eq:ellipticity}
\frac{1}{M} | P|^2  \le \sigma_i \, P\cdot  P\le M | P|^2 \qquad \text{for all $P \in \R^{dN^k}, \quad  i \in \{1,2\}$}.
\end{equation}
For a fixed Borel set $A \subset \R^N$, define the two-point valued tensor
\begin{equation}\label{eq:definition}
\sigma_A(x) \coloneqq  \mathds {1}_A \sigma_1 + \mathds1_{(\Omega \setminus A)}\sigma_2.
\end{equation} 

We consider a $k$-homogeneous linear differential operator
$\A : \Lrm^2\left(\Omega;\R^d\right) \to \Wrm^{-k,2}(\Omega;\R^{dN^k})$ of gradient form (see Definition \ref{gto} in Section \ref{preliminaries}).
As a consequence of the definition of operators of gradient form, the following equation 
 \begin{equation}\label{state eq}
\A^* (\sigma_A \A u) = f \quad \text{in $\mathcal D'(\Omega;\R^d)$}, \qquad    u \in \Wrm_0^{\A}(\Omega) \subset \Wrm^{k,2}_0(\Omega;\R^d),\end{equation}
has a unique solution (cf. Theorem \ref{existence}). We will refer to equation (\ref{state eq}) as the {\it state constraint} and we will denote 
by $w_A$ its unique solution. 

It is a physically relevant question to ask which designs have the least dissipated energy.
To this end, consider the energy defined as
 \[A \; \mapsto \; E(A) \coloneqq \int_\Omega fw_A \, \dx x~ + ~\text{Per}(A;\overline\Omega) \qquad \text{among Borel subsets $A$ of $\R^N$}. \footnote{Here, $\Per(A;\overline \Omega) = |\mu_A|(\overline \Omega)$, where $\mu_A$ is the Gauss-Green measure of $A$; see Section \ref{sec:gmt}.}\]
We will be interested in the optimal design problem with Dirichlet boundary conditions on sets: 
\begin{equation}\label{intro1}
\text{minimize} \qquad \left\{\; E(A) \; : \; A \subset \R^N \; \text{ is a Borel set}, ~A\cap \Omega ^c \equiv A_0 \cap \Omega^c \: \right\},\footnote{Due to the nature of the problem, we cannot replace $\Per(A;\overline \Omega)$  with $\Per(A;\Omega)$ in $E(A)$ because it possible that minimizing sequences tend to accumulate perimeter in $\partial \Omega$.}
\end{equation}
where $A_0 \subset \R^N$ is a set of locally finite perimeter.

Most attention has been drawn to the case where designs are mixtures of two well-ordered materials. The presentation given here places no comparability hypotheses on $\sigma_1$ and $\sigma_2$. 
Instead, we introduce a weaker condition on the decay of generalized minimizers of a double-well problem. Our technique also holds under various constraints other than Dirichlet boundary conditions; in particular, any additional cost that scales as $O(r^{N -1 + \varepsilon})$. For example, a constraint on the volume occupied by a particular material (cf. \cite{Lin93,Carozza14,DPFig15}). Lastly, we remark that our technique is robust enough to treat models involving the {\it maximization} of dissipated energy.

\subsection{Main results and background of the problem}
Existence of a minimizer of \eqref{intro1} can be established by standard methods. We are interested in proving that a solution pair $(w_A,A)$ enjoys  better 
{\it regularity} properties 
than the ones needed for existence. The notion of regularity for a set $A$ will be understood as the local regularity of
$\partial A$ seen as
a submanifold of $\R^N$, whereas the notion of regularity for $w_A$ will refer to its differentiability and integrability properties.  

It can be seen from the energy, that the deviation from being a perimeter minimizer for a solution $A$ of problem (\ref{intro1})
is bounded by the dissipated energy. Therefore, one may not expect better regularity properties for $A$ than the ones 
for perimeter minimizers; and thus, one may only expect regularity up to singular set (we refer the reader to~\cite{DEGio61,Alm68} for classic results, see also \cite{DPFig15} for a partial regularity result in a similar setting to ours). 

Since a constrained problem may be difficult to treat, we will instead consider an equivalent variational
unconstrained problem by introducing a multiplier as follows. Consider the saddle point problem 
\begin{equation*}\label{P}
 \inf_{A\subset \Omega} \sup_{u \in \Wrm^{\A}_0(\Omega)} I_\Omega(u,A), \tag{P}
\end{equation*}
where
\[
I_\Omega(u,A) \coloneqq \int_\Omega 2fu \, \dx x~ -~ \int_\Omega \sigma_A \A u \cdot \A u \, \dx x~ +~ \text{Per}(A;\overline \Omega).
\]
Our first result shows the equivalence between problem \eqref{P} and the minimization problem (\ref{intro1}) under the state constraint (\ref{state eq}): 
\begin{theorem}[existence]\label{existence} There exists a solution $(w,A)$ of 
problem \eqref{P}. Furthermore, there is a one to one correspondence 
\[
 (w,A) \mapsto (w_A,A) 
\]
between solutions to problem \eqref{P}  and the minimization problem
(\ref{intro1}) under the state constraint (\ref{state eq}).
\end{theorem}

We now turn to the question of regularity. Let us depict an outline of the key steps and results obtained in this regard. The Morrey space $\Lrm^{p,\lambda}(\Omega;\R^d)$ is the subspace of $\Lrm^p(\Omega;\R^d)$ for which the semi-norm
\[
[u]^p_{\Lrm^{p,\lambda}(\Omega)} \coloneqq \sup\left\{ \frac{1}{r^\lambda}\int_{B_r(x)} | u |^p \dd y: B_r(x) \subset \Omega \right\}, \qquad 0 < \lambda \le N,
\]
is finite.

The {\it first step} in proving regularity for solutions $(w,A)$ consists in proving a critical $\Lrm^{2,\,N-1}$ local estimate
for $\A w$. This estimate arises na\-tu\-ra\-lly since we expect a kind of balance between $\int_{B_r(x)} \sigma_A \A w \cdot \A w \dd y$ and the perimeter part $\Per(A;B_r(x))$ that scales as $r^{ N-1}$ in balls of radius $r$. 

To do so, let us recall a related relaxed problem. 
 As part of the assumptions on $\A$ there must exist a constant rank, $m$-order differential operator 
 $\B :  \Lrm^2(\Omega;Z) 
 \to \Wrm^{-m,2}(\Omega;\R^n)$ with 
 $\text{Ker}(\B) = \A[\Wrm^{\A}(\Omega)]$. \footnote{Here, $\Wrm^{\A}(\Omega) = \big\{u \in \Lrm^2(\Omega;\R^d) : \A u \in \Lrm^2(\Omega;\R^{dN^k}) \big\}$ is the $\A$-Sobolev space of $\Omega$.} 
 It has been shown by Fonseca and M\"uller \cite{FonsecaMuller99}, that a necessary and sufficient condition for the lower semi-continuity of 
 integral energies with superlinear growth under a constant rank differential constraint $\B v = 0$ is the $\B$-quasiconvexity of the integrand.
 In this context, the $\B$-free quasiconvex envelope of the double-well $W(P)\coloneqq \min\{\sigma_1 \, P \cdot P,\,\sigma_2\, P \cdot P\}$, at a point 
 $P\in Z \subset \R^{dN^k}$, is given by
 \begin{align*}
  Q_{\B} W(P) & \coloneqq \inf\setBB{\int_{[0,1]^N} W(P + v(y)) \dd y}{ \\ & v \in \Crm^\infty_{\text{per}}\big([0,1]^N;Z\big), \B v = 0 \; \text{and} \,
  \int_{[0,1]^N} v(y) \dd y = 0}.
 \end{align*}
The idea is to get an $\Lrm^{2,\,N-1}$ estimate by transferring the regularizing effects from generalized minimizers of the energy
$u \mapsto \int_{B_1} W(\A u)$ onto our original problem. In order to achieve this, we use a $\Gamma$-convergence argument with respect to
a perturbation in the interfacial energy from which the next result follows:
\begin{theorem}[upper bound]\label{thm2} 
Let $(w,A)$ be a variational solution of problem \eqref{P}.
Assume that
the higher integrability condition 
 \begin{equation}\label{eq:assumption}
  [\A \tilde u]^2_{\Lrm^{2,N-\delta}(B_{1/2})} \le c\|\A \tilde u \|_{\Lrm^2(B_1)}^2, \quad \text{for some $\delta \in [0,1)$ and some positive constant $c$}, \tag{Reg}
 \end{equation}
 holds for local minimizers of the energy $u \mapsto \int_{B_1} Q_{\B} W(\A u)$, where $u \in \Wrm^{\A}(\Omega)$.
  Then, for every compactly contained set $K \subset \subset \Omega$, there exists a positive constant $ \Lambda_K$ such that
\begin{equation}\label{upper}
 \int_{B_r(x)} \sigma_A \A w \cdot \A w \; \dd y \; + \; \Per(A;B_r(x)) \le \Lambda_K r^{N-1},
\end{equation}
\text{for all $x \in K$ and every $r \in (0,\dist(K,\partial \Omega))$}.
\end{theorem}
\begin{remark}[well-ordering assumption]
 If $\sigma_1, \sigma_2$ are well-ordered, say $\sigma_2 - \sigma_1$ is positive definite, then $Q_{\B}W$ is precisely the quadratic form $\sigma_2\, P \cdot P$. Due to standard elliptic regularity results (cf. Lemma \ref{lem:cosntant}), estimate (\ref{eq:assumption}) holds for $\delta =0$; therefore, assuming that the materials are well-ordered is a sufficient condition for the higher integrability assumption (\ref{eq:assumption}) to hold.
\end{remark}
\begin{remark}[non-comparable materials]
 In dimensions
 $N = 2, 3$ and restricted to
 the setting $\A = \nabla$, $d = 1$, condition  (\ref{eq:assumption}) is strictly weaker than assuming the materials to be well-ordered. Indeed, one can argue by a Moser type iteration as in 
 \cite{CarMull02} to
 lift the regularity of minimizers. For higher-order gradients or in the case of systems
 it is not clear to us whether assumption \eqref{eq:assumption} is equivalent to the well-ordering of the materials. 
\end{remark}

The {\it second step}, consists of proving a {\it 
discrete monotonicity}
for the excess of the Dirichlet energy on balls under a low perimeter density assumption. More precisely, on the function that assigns 
\[
 r \quad \mapsto \quad \frac{1}{r^{N-1}}\int_{B_r(x)} |\A w|^2 \, \dx x, \qquad x \in \partial A, ~r > 0.
\]
The discrete monotonicity of the map above, together with the upper bound estimate (\ref{upper}), will allow us to prove a local 
{lower bound} $\lambda_K$ on the density of 
the perimeter:
\begin{equation*}\label{lower}
 \tag{LB} \frac{\Per(A;B_r(x))}{r^{N-1}} \ge \lambda_K \qquad \text{for every $x \in (K\cap \partial A)$, and every $0 < r \le r_K$}.
\end{equation*}

As usual, the lower bound on the density of the perimeter is the cornerstone to prove regularity of almost perimeter minimizers. In fact, once the estimate \eqref{lower} is proved we simply apply the excess improvement results of \cite[Sections 4 and 5]{Lin93} to obtain our main result:

\begin{theorem}[partial regularity]\label{main} Let $(w,A)$ be a saddle point of problem \eqref{P} in $\Omega$. 
	Assume that the operator $P_Hu = \A^*(\sigma_H \A u)$ is hypoelliptic and regularizing for the half-space problem (see properties \eqref{eq:hypo}-\eqref{eq:halfplane}), and that the higher integrability \eqref{eq:assumption} holds. Then there exists a positive constant $\eta \in (0,1]$ depending only on $N$ such that 
	\[
	\mathcal H^{N-1}((\partial A \setminus \partial^* A) \cap \Omega) = 0, \quad \text{and} \quad \partial^* A \quad \text{is an open 
		$\Crm^{1,\eta/2}$-hypersurface in $\Omega$}.
	\]
	Moreover if $\A$ is a first-order partial differential operator, then $\A w \in \Crm_{\textnormal{loc}}^{0,\eta/8}(\Omega \setminus (\partial A \setminus \partial^* A))$;
	and hence, the trace of $\A w$ exists on either side of $\partial^* A$.
\end{theorem}

Let us make a quick account of previous results. To our knowledge, only optimal design problems modeling the maximal dissipation of energy have been treated. 

In \cite{AmbrosioButtazzo93} 
Ambrosio 
and Buttazzo considered the case where $\A = \nabla$ is the gradient operator for scalar-valued ($d = 1$) functions and 
where $\sigma_2 \ge \sigma_1$ in the sense of quadratic forms. 
The authors proved existence of solutions and showed that, up to choosing a good representative, the topological boundary is the closure of the reduced boundary and
$\mathcal H^{N-1}(\partial A \setminus \partial^*A) = 0$.
Soon after, Lin \cite{Lin93}, and Kohn and Lin \cite{LinKohn99} proved, in the same 
case,
that $\partial^* A$ is an open $\Crm^1$-hypersurface. From this point on, there have been
several contributions aiming to discuss the optimal regularity of the interface for this particular case. In this regard
and in dimension $N =2$, 
Larsen \cite{Larsen13} proved that connected components of $A$ are $\Crm^1$ away from the boundary.
In arbitrary dimensions, Larsen's argument cannot be further generalized because it relies on the fact that convexity and positive curvature 
are equivalent in dimension $N =2$.
During the time this project was developed, we have learned that Fusco and Julin \cite{FuscoJulin15} found a different proof for the same results as stated in
\cite{Lin93}; besides this, De Philippis and Figalli \cite{DPFig15} recently obtained an improvement on the dimension of the singular set ($\partial^*A \setminus \partial A$).

The paper is organized as follows. In the beginning of Section \ref{preliminaries} we fix notation and discuss some facts of linear operators, Young measures and sets of finite perimeter. 
We also give the precise
definition of gradient type operators and include a compensated compactness result that will be employed throughout the paper. In Section \ref{exsol} 
we show the equivalence of the constrained problem (\ref{state eq})-(\ref{intro1}) and the unconstrained problem \eqref{P} (Theorem \ref{existence}). In the first part of Section \ref{CE} we shortly discuss how the higher integrability assumption (\ref{eq:assumption}) holds for various operators of gradient form. The rest of the section is devoted to the proof of the upper bound (\ref{upper}). Section \ref{LBo} is devoted to the proof of the lower bound estimate (\ref{lower}). Finally, in Section \ref{FE} we recall the
flatness excess 
improvement \cite{Lin93} from  which Theorem \ref{main} easily follows.

\section{Notation and preliminaries}\label{preliminaries}

We will write $\Omega$ 
to represent a non-empty, open, bounded subset of $\R^N$ with Lipschitz boundary $\partial \Omega$. The use of
capital letters $A, B, \dots,$ will be reserved to denote Borel
subsets of $\R^N$ and we will write $\mathfrak B(\Omega)$ to denote the Borel $\sigma$-algebra in $\Omega$. 

 The letters $x,y$ will denote points in $\Omega$; while $z \in \R^d$ and $P \in  \R^{dN^k}$ will be reserved for vectors and arrays in Euclidean space.  The Greek letters $\varepsilon, \delta, \rho$ and $\gamma$ shall be used for general smallness or scaling constants. 
We follow Lin's convention in \cite{Lin93}, bounding constants will be generally denoted by $c_1\ge c_2 \ge
\dots$, while smallness and decay constants will be usually denoted by $\varepsilon_1\ge \varepsilon_2\ge
\dots$, and $\theta_1\ge\theta_2\ge\dots,$ respectively.  Let us mention that in proving regularity results one may often find it impractical to keep track of numerical constants due to the large amount of parameters; to illustrate better their uses and dependencies we have included a glossary of constants at the end of the paper. 

It will often be useful to 
write a point $x \in \R^{N} = \R^{N-1} \times \R$ as $x = (x',x_{N})$, in the same fashion we will also write
$\nabla = (\nabla',\partial_{N})$ to decompose the gradient operator.
The bilinear form $\R^p \times \R^p \to \R : (x,y) \mapsto x \cdot y$ will stand for the standard inner product between two points
while we will use the notation $| x | := \sqrt{x \cdot x}$ to represent the standard $p$-dimensional Euclidean norm.
To denote open balls centered at
a point $x$ with radius $r$ we will simply write $B_r(x)$.
%

We keep the standard notation for $\Lrm^p$ and $\mathrm \mathrm \mathrm W^{l,p}$ spaces. We write 
$\Crm^l(\Omega;Z)$, and $\Crm^l_c(\Omega;\R^d)$ to denote the spaces of functions with values in $\R^d$ and
with continuous $l$-th derivative, and its subspace of functions compact support respectively. Similar notation
stands for $\mathcal M(\Omega;\R^d)$, the space of bounded Radon measures in $\Omega$; and $\mathcal D(\Omega;\R^d)$, the space of smooth functions in $\Omega$ with compact support. 
For $X$ and $Y$ Banach spaces, the standard 
pairing between $X$ and $Y$ will be denoted by
$\langle \cdot,\cdot\rangle : X \times Y \to \R : (u,v) \mapsto \langle u, v \rangle$.

\subsection{Operators of gradient form}\label{subsec:gto} We introduce an abstract class of linear differential operators $\A : \Lrm^2(\Omega;\R^d) \to \Wrm^{-k,2}(\Omega;\R^{dN^k})$. 
This class contains scalar- and vector-valued gradients, higher gradients, and symmetrized gradients among its elements. 
The motivation behind it
is that we may treat different models by employing a general and neat abstract setting.
At a first glance this framework may appear too sterile; however, this definition is only meant to capture some of the essential regularity and rigidity properties of gradients.

Let $\A : \Lrm^2(\Omega;\R^p) \to \Wrm^{-k,2}(\Omega;\R^q)$ be a $k$-th order homogeneous partial differential operator of   the form  
	\begin{equation}\label{def:ogf}
	\A = \sum_{|\alpha| = k} A_\alpha\partial^\alpha,
	\end{equation}
	where $A_\alpha \in \text{Lin}(\R^p;\R^{q})$, and
 $\partial^\alpha = \partial_1^{\alpha_1}\dots \partial_N^{\alpha_N}$ for every 
	multi-index $\alpha = (\alpha_1,\dots,\alpha_N) \in (\mathbb N \cup \{0\})^N$ with $|\alpha| \coloneqq |\alpha_1| + \dots |\alpha_N|$. 
We define the $\A$-Sobolev of $\Omega$ as
\begin{gather*}
\Wrm^{\A}(\Omega) \coloneqq \bigg\{u \in \Lrm^2(\Omega;\R^d) \; : \; \A u \in \Lrm^2(\Omega;\R^{dN^k})\bigg\}
\end{gather*}
endowed with the norm  $\|u\|^2_{\Wrm^{\A}(\Omega)} \coloneqq \| u \|^2_{\Lrm^2(\Omega)} + \| \A u \|^2_{\Lrm^2(\Omega)}$. We also define the $\A$-Sobolev space with zero boundary values in $\partial \Omega$ by letting
\[
\Wrm^{\A}_0(\Omega) \coloneqq \text{cl}\bigg\{\Crm^\infty_c(\Omega;\R^d), \| \cdot \|_{\Wrm^{\A}(\Omega)}\bigg\}.
\]
The principal symbol of $\A$ is the positively $k$-homogeneous map defined as
\begin{equation*}\label{rank}
\xi \mapsto \mathbb A(\xi) := \sum_{|\alpha| = k} \xi^\alpha A_\alpha \in \text{Lin}(\R^p,\R^{q}), \qquad \xi \ \in \R^N,
\end{equation*}
where $\xi^\alpha = \xi_1^{\alpha_1}\cdots\xi_N^{\alpha_N}$. One says that $\A$ has the constant rank property if 
there exists a positive integer $r$ such that 
\begin{equation*}\label{rank}
\text{rank}\,({\mathbb A(\xi)})= r \qquad \text{for all $\xi \in \R^N \setminus \{0\}$}.\tag{$\dagger$}
\end{equation*}
	
\begin{definition}[Operators of gradient form]\label{gto} Let $\A$ a homogeneous partial differential operator as in \eqref{def:ogf} with $p = d$ and $q = dN^k$.
We say that $\A$ is an operator of gradient form if the following properties hold:

\begin{description}
	\item[1. {\bf Compactness:}] There exists a positive constant $C(\Omega)$ for which
	\begin{equation}\label{def:compactness}
\| \varphi\|^2_{\mathrm W^{k,2}(\Omega)} \le	C(\Omega)\bigg(\| \varphi \|_{\Lrm^2(\Omega)}^2 + \|\A \varphi\|^2_{\Lrm^2(\Omega)}\bigg)
	\end{equation}
	for all $\varphi \in \Crm^\infty(\overline\Omega;\R^d)$. Even more, for every $u \in \mathrm W^{\A}(\Omega)$ the following Poincar\'e inequality holds:
	\begin{equation}\label{poincare1}
	\inf\big\{ \;\| u - v\|^2_{\mathrm W^{k,2}(\Omega)} \; :\; v \in \mathrm W^{\A}(\Omega), \A v = 0\; \big\} \le C(\Omega)\| \A u \|^2_{\Lrm^2(\Omega)}.
	\end{equation}
	
	\item[2. {\bf Exactness:}] There exists an $m$-th homogeneous
partial differential operator
\begin{equation}\label{form}
  \mathcal B \coloneqq  \sum_{|\alpha| = m} B_\alpha \partial^\alpha,
\end{equation}
with coefficients $B_\alpha \in \text{Lin}(Z;\R^n)$ for some positive integer $n$ and a subspace $Z$ of $\R^{dN^k}$, such that for every open and simply connected subset $\omega \subset \Omega$ we have the property
  \[
  \big\{\A u : u \in \Wrm^{\A}(\omega) \big\}= \big\{ v \in \Lrm^2(\omega;Z) : \B v = 0 \; \text{in $\mathcal D'(\omega;\R^n)$}\big\}.
  \]
\end{description}
We write $\A^*$ to denote the $\Lrm^2$-adjoint of $\A$, which is given by
\[
\A^* \coloneqq (-1)^k\sum_{|\alpha| = k} A^T_\alpha\partial^\alpha.
\] 
\end{definition}

\begin{remark}[constant rank]\label{constante}
Let $\A$ and $\B$ be two linear differential operators satisfying an exactness property as in Definition \ref{gto}. Then
both
operators $\A$ and $\B$ have the constant rank property \eqref{rank}.
This follows from the lower semi-continuity of the rank in any subspace of matrices. 
\end{remark}

\begin{remark}[rigidity]\label{rem:wave} The wave cone of an operator $\A$ of the form \eqref{def:ogf} which is defined as
	\[
	\Lambda_{\A} \coloneqq \bigcup_{|\xi|  = 1} \ker(\mathbb A(\xi)) \subset \R^{p},
	\]
 contains the admissible amplitudes in Fourier space for which concentration and oscillation behavior is allowed under the constraint $\A u = 0$. As in the case of gradients, it can be seen from the compactness assumption in Definition \ref{gto} that the wave cone $\Lambda_{\A}$ of a gradient operator $\A$ is the zero space. In particular, there exists a positive constant $\lambda$ (depending only on the coefficients of $\A$) such that
 \begin{equation}\label{eq:strongly elliptic systems}
 |\mathbb A(\xi) z |^2 \ge \lambda |\xi|^{2k} |z|^2 \qquad \text{for all $\xi \in \R^N \setminus \{0\}$ and all $z \in \R^d$}. 
 \end{equation}
\end{remark}

\begin{remark}[Poincar\'e inequality II]\label{poincare22} It follows from the definition of $\Wrm_0^{\A}(\Omega)$ and the compactness assumption of $\A$ that $\Wrm^{\A}_0(\Omega)\subset \Wrm^{k,2}_0(\Omega;\R^d)$. In particular, $\ker(\A) \cap \Wrm^{\A}_0(\Omega) = \{0\} \subset \Lrm^2(\Omega;\R^d)$ and $\A[\Wrm^{\A}_0(\Omega)]$ is closed in the $\Lrm^2$ norm. Thus, by~\cite[Theorem 2.21]{BrezisBook83}, there exists
 a constant\footnote{Possibly abusing the notation, we will denote by $C(\Omega)$ the Poincar\'e constants from Definition \ref{gto} and Remark \ref{poincare22}.}   $C(\Omega)$\begin{equation}\label{poincare2}
\| u \|^2_{\Lrm^{2}(\Omega)} \le C(\Omega) \|\A u\|^2_{\Lrm^2(\Omega)} \qquad \text{for all $u \in \Wrm^{\A}_0(\Omega)$}.
\end{equation}
\end{remark}

\subsubsection{Elliptic regularity}

Let $\A$ be an operator of gradient form as in Definition \ref{gto} and let $\mathbf \sigma \in \Lrm^\infty(\Omega;\R^{dN^k})$ be a tensor of variable coefficients satisfying the strong pointwise G\aa rding inequality (see \eqref{eq:ellipticity})
\begin{equation}\label{eq:strongg}
\frac{1}{M} |P|^2 \le \mathbf \sigma(x) \, P \cdot P \le M|P|^2 \qquad \text{for almost every $x \in \Omega$ and every $P \in \R^{dN^k}$}. 
\end{equation}
If we define
	\[
	\mathbf A_{\beta\alpha}^{ij} \coloneqq (A_\alpha)_{i\beta,j} \qquad \text{for $|\alpha| = |\beta| = k$, and $1 \le i,j \le d$},
	\]
	then we may write 
	\begin{equation}\label{def:bilinear}
	\A \varphi = \mathbf A \nabla^k \varphi \qquad \text{for every $\varphi \in \Crm^k(\overline\Omega;\R^d)$}.
	\end{equation}
It is easy to verify, using the compactness assumption of $\A$, that $\mathbf C \coloneqq (\mathbf A^T \sigma \;\mathbf A)$ satisfies the weak G\aa rding inequality 
\begin{equation}\label{eq:wgarding}
 \langle \mathbf C \; \nabla^k \varphi, \nabla^k\varphi \rangle \ge \left(\frac{1}{MC}\right) \| \nabla^k \varphi \|^2_{\Lrm^2(\Omega)} - \left(\frac{1}{M}\right) \|\varphi\|^2_{\Lrm^2(\Omega)},
\end{equation}
where $C = C(\Omega)$ the constant in the compactness assumption of Definition \ref{gto}; 
for all smooth, $\R^d$-valued functions $\varphi$ in $\overline \Omega$. 
 
\begin{lemma}[Caccioppoli inequality]\label{rem:cac}
	Let $\mathbf \sigma \in \Lrm^\infty(\Omega;\R^{dN^k})$ satisfy the strong pointwise G\aa rding inequality \eqref{eq:strongg} and let $w \in \mathrm W^{\A}(\Omega)$ be a solution of the state equation
	\[
	\A^* (\sigma \A u)= 0 \qquad \text{in $\mathcal D'(\Omega;\R^d)$}.
	\]
	Then there exists a positive constant $C$ depending only on $M,N, \sigma$ and $\A$ such that
	\begin{gather*}
\int_{B_r(x)} |\nabla^k w|^2 \dd x \;  \le \; \frac{ C}{(R- r)^{2k}} \;  \int_{B_R(x)} |w|^2 \dd x
\qquad \text{	for every $B_r(x) \subset B_R(x) \subset \Omega$}.
	\end{gather*}
\end{lemma}
\begin{proof}
We may re-write $\A^* (\sigma \A u)$ as the elliptic operator in divergence form
\[
(-1)^k\sum \partial^\beta(\mathbf C^{ij}_{\beta \alpha} \partial^\alpha u^j),
\]
for coefficients $\mathbf C = (\mathbf A^T \sigma \mathbf A)$ satisfying a weak G{\aa}rding inequality as in \eqref{eq:wgarding}.
The assertion then follows from Corollary 22 in \cite{barton2014gradient}.

\end{proof}

Using Lemma \ref{rem:cac} one can show, by classical methods, the following lemma on the regularizing properties of elliptic operators with constant coefficients:

\begin{lemma}[constant coefficients]\label{lem:cosntant} Let $\A$ be an operator of gradient form and let $\sigma_0 \in \textnormal{Lin}(\R^d;\R^{dN^k})$ be a tensor satisfying the strong G{\aa}rding inequality \eqref{eq:strongg}. Then the operator
\[
L_{\sigma_0}u \coloneqq \A^*(\sigma_0 \A u)
\]
is hypoelliptic in the sense that if $\Omega$ is open and connected, and $w \in \Lrm^2(\Omega;\R^d)$, then
\[
L_{\sigma_0}w = 0  \quad \Rightarrow \quad w \in \Crm_{\textnormal{loc}}^\infty(\Omega;\R^d).
\]
Furthermore, there exists a constant $c = c(M,N) \ge 2^N$ such that
\[
\frac{1}{\rho^N}\int_{B_\rho(x)} |\nabla^k u|^2 \dd x \le 
\frac{c}{r^N}\int_{B_r(x)} |\nabla^k u|^2 \dd x  
\qquad \text{for all $0 < \rho \le \frac{r}{2}$,}
\]
\[
\frac{1}{\rho^N}\int_{B_\rho(x)} |\A u|^2 \dd x \le 
\frac{c}{r^N}\int_{B_r(x)} |\A  u|^2 \dd x  
\qquad \text{for all $0 < \rho \le \frac{r}{2}$,}
\]
for every $B_r(x) \subset \Omega$.
\end{lemma}

\subsubsection{Examples}\label{examples} Next, we gather some well-known differential structures that fit into
the definition of operators of gradient form.

\begin{itemize}
 \item [(i)] {\bf Gradients.} Let $\A : \Lrm^2(\Omega;\R^d) \to \Wrm^{-1,2}(\Omega;\R^{dN}) : u \mapsto (\partial_j u^i)$ for $1 \le i \le d$ and $1 \le j \le N$. In this case
 \[
 A_j \, z = z \otimes \mathbf e_j \quad \text{for every $z \in \R^d$}.
 \]
 Hence, $\Wrm^{\A}(\Omega) = \Wrm^{1,2}(\Omega;\R^d)$ and the compactness property is a consequence of the classical Poincar\'e inequality on $\Omega$.

The exactness assumption is the result of the  characterization of gradients 
via curl-free vector fields.

 Let 
 $\B : \Lrm^2(\Omega;\R^{dN}) \to \Wrm^{-1,2}(\Omega;\R^{dN^2})$ be the curl operator 
 \[
\B v = (\curl(v^i))_i\coloneqq \left(\partial_l v_{ir} - 
 \partial_r v_{il} \right)_{ilr} \qquad 1 \le i \le d, \quad 1 \le l,r \le N,
 \]
 then condition (\ref{form}) is fulfilled for $\B = \sum_{j=1}^N B_j \partial_i$ with coefficients
 \[
       (B_j)_{ilr,pq} = \delta_{ip}(\delta_{jl}\delta_{rq} - \delta_{jr}\delta_{lq})
    \qquad 1\le l,r,q \le N, \quad 1 \le i,p \le d.
         \]
Observe that $\B v = 0$ if and only if curl $v^i = 0$, for every $1 \le i \le d$; or equivalently,
$ v^i = \nabla u^i$ for some function $ u^i : \Omega \subset \R^N \to \R$, for every $1 \le p \le d$ (as long as $\Omega$ is 
simply connected). Hence,
\[
\big\{ \nabla u : u \in \Wrm^{1,2}(\omega)\big\} = \big\{v \in \Lrm^2(\omega;\R^{dN}) : \B v = 0\big\},
\]
text{for all Lipschitz, and simply connected $\omega \subset \subset\Omega$}.
 \item [(ii)] {\bf Higher gradients.} Let $\A : \Lrm^2(\Omega) \to \Wrm^{-k,2}(\Omega;\R^{N^k})$ be the linear operator given by
 \[
 u \mapsto \partial^\alpha u, \qquad \text{where  $|\alpha| = k$}.  
 \]
Compactness is similar to the case of gradients.

 We focus on the exactness condition: Let 
 \[\B^k : \Lrm^2(\Omega;\text{Sym}(\R^{N^k})) \to \Wrm^{-1,2}(\Omega;\R^{N^{k+1}})\] be the curl operator on symmetric functions defined by the coefficients
 \[
  (B^k_j)_{pq\beta_2\dots \beta_k,\alpha_1\dots \alpha_k} := \left(\delta_{jp}\delta_{\alpha_1 q}\prod_{h= 2}^k\delta_{\alpha_h\beta_h}
  - \delta_{jq}\delta_{\alpha_1 p}\prod_{h= 2}^k\delta_{\alpha_h\beta_h}\right), \]
  where $1 \le p,q,\beta_h,\alpha_h \le N, \quad h\in\{2,\cdots,k\}$.

 We write 
 \[
  \B^k v \coloneqq \sum_{i = 1}^N B^k_j \, \partial_j v, \qquad v : \Omega \subset \R^N \to \text{Sym}(\R^{N^k}).
 \]
 
It easy to verify that $\B^k v = 0$ if and only if 
\[
\curl((v_{p\alpha'})_p) = 0 \qquad \text{for all $|\alpha'| = k-1$}.
\]
If $\Omega $ is simply connected, then there exists a function $u^{\alpha'} : \Omega \to \R$ such that 
$v_{p\alpha'} = \partial_p u^{\alpha'}$ for every 
$|\alpha'| = k-1$. 
Using the symmetry of $v$ under the permutation of its coordinates one can further deduce
 the existence of a function $u_k : \Omega \to \text{Sym}(\R^{N^{k-1}})$ with
  \[
  v = \nabla u_k \quad \text{and \quad $(u_k)_{\alpha'} = u^{\alpha'}$}.
  \]
 Moreover, $\B^{k-1} u_k = 0$. By induction one obtains that
\[
 v = \nabla^{k}u_0 \qquad \text{for some function $u_0:\Omega \subset \R^N \to \R$}.
\]

\item [(iii)] {\bf Symmetrized gradients.} Let 
$\mathcal E : \Lrm^2(\Omega;\R^N) \to \Wrm^{-1,2}(\Omega;\text{Sym}(\R^{N^2}))$ be the linear operator given by
 \[
 u \mapsto \mathcal E u\coloneqq  \frac{1}{2}(\partial_j u^i + \partial_i u^j)_{ij}, \qquad \text{for }\; 1\le i,j \le N.  
 \]
The compactness property is a direct consequence of Korn's inequality. 
Consider the second-order homogeneous differential 
operator $\B: \Lrm^2(\Omega;\text{Sym}(\R^{N^2})) \to \Wrm^{-2,2}(\Omega;\R^{N^3})$ defined in the following way
 \[
 \B v = \curl \, (\curl (v)) =  \left(\frac{\partial^2 v_{ij}}{\partial x_i \partial x_l} + \frac{\partial^2 v_{il}}{\partial x_i \partial x_j}
 -\frac{\partial^2 v_{ii}}{\partial x_j \partial x_l} -\frac{\partial^2 v_{jl}}{\partial x_i \partial x_i} \right)_{1 \le i,j,l\le N}.\footnote{Here, $\mathscr B$ is a second order operator expressing the Saint-Venant compatibility conditions.}
 \]
Then $\B v = 0$, if and only if $v = \mathcal Eu$ for some $u \in \mathrm W^{1,2}(\Omega;\R^N) = \Wrm^{\mathcal E}(\Omega)$.
\end{itemize}

\begin{remark}
	In the previous examples, we have omitted the characterization of higher gradients of vector-valued functions; however, the ideas remain the same as in the examples (i) and (ii).
\end{remark}

\begin{remark}[two-dimensional elasticity]\label{rem:plate} In dimension $N = 2$ and provided that $\Omega$ is simply connected, the fourth-order equation for pure bending of a thin plate given by
	\[
	\nabla \cdot \nabla \cdot (\mathbf D(x) \nabla^2 u(x)) = 0 \qquad \text{for $u \in W^{2,2}(\Omega)$}
	\] 
	is equivalent to the in-plane elasticity equation
	\[
	\nabla \cdot ( \mathbf S(x) \mathcal E w(x)) = 0 \qquad \text{where $w \in W^{1,2}(\Omega;\R^2)$},
	\]
	for some tensor $\mathbf S$ such that $\mathbf D = (\mathbf R_\perp \mathbf S^{-1} \, \mathbf R_\perp)$, and where $ \mathbf R_\perp$ is the fourth-order tensor whose action is to rotate a second-order tensor by $90^\circ$ (see, e.g.,  \cite[Chapter 2.3]{MiltonBook}). Furthermore,
	\[
	\mathbf S(x) \mathcal E w(x) = \mathbf R_\perp \nabla^2 u(x) \quad \text{and} \quad \nabla \cdot \nabla \cdot (\mathbf R_\perp \mathcal E w(x)) = 0.
	\]
	  For this reason, when working with the linear equations for pure bending of a thin plate we may indistinctly use regularizing properties of any of the equations above in the portions where $\mathbf D$ is regular. 
	
\end{remark}

\subsection{Compensated compactness} The following theorem is a generalized version of the well-known div-curl Lemma. 
\begin{lemma}\label{compensated} Let $\A$ be a $k$-th order operator of gradient form and 
 let $\{\sigma_h\}\subset \Lrm^2(\Omega;\R^{dN^k} \otimes \R^{dN^k})$ be a sequence of strongly elliptic tensors as in \eqref{eq:strongg}. Assume also that $\{u_h\} \subset \Wrm^{\A}(\Omega)$ and $
 \{f_h\} \subset \Wrm^{-k,2}(\Omega;\R^d)$ are sequences for which 
 \[
  \A^*(\sigma_h \A u_h) = f_h \quad \text{in $\mathcal D'(\Omega;\R^d)$, \; for every $h \in \mathbb N$}.
 \]
Further assume there exist $\sigma \in \Lrm^2(\Omega;\R^{dN^k} \otimes \R^{dN^k})$,
$u \in \Wrm^{\A}(\Omega)$, and $f \in \Wrm^{-k,2}(\Omega;\R^d)$  for which
\begin{gather*}
 \A u_h \wc \A u \quad \text{in $\Lrm^2(\Omega;\R^{dN^k})$}, \qquad f_h \to f \quad \text{in $\Wrm^{-k,2}(\Omega;\R^d)$}, \\ \text{and} \quad
 \sigma_h \to \sigma \quad \text{in $\Lrm^2(\Omega;\R^{dN^k} \otimes \R^{dN^k})$}.
\end{gather*}
Then, 
\begin{gather*}
\A^*(\sigma \A u) = f \quad  \text{in $\mathcal D'(\Omega;\R^d)$},\\
\sigma_h \A u_h \cdot \A u_h \to \sigma \A u \cdot \A u \quad \text{in $\mathcal D'(\Omega)$}. 
\end{gather*}
In particular,
\[
 \A u_h \to \A u\quad \text{in $\Lrm^2_{\textnormal{loc}}(\Omega;\R^{dN^k})$}.
\]
\end{lemma}
\begin{proof}
 For simplicity we denote $\tau_h:= \sigma_h\A u_h, \tau := \sigma \A u$. 
 It suffices to observe that $\tau_h\wc \tau$ in $\Lrm^2$ to prove that 
 \[
  \A^* \tau = f \quad \text{in $\mathcal D'(\Omega;\R^d)$}.
 \]

The strong convergence on compact subsets of $\Omega$ requires a little bit more 
effort. 
 Considering that $\A$ is a $k$-th order linear differential operator, we may find constants $c_{\alpha\beta}$ with
 $|\alpha| + |\beta| \le k, |\beta | \ge 1$ such that
 \[
  \A(u_h \varphi) = (\A u_h) \varphi + \sum_{\alpha,\beta} c_{\alpha\beta} \partial^\alpha u_h \partial^\beta 
  \varphi \in \Lrm^2(\Omega;\R^d) \qquad \forall\;\varphi \in \mathcal D(\Omega), \forall \; h\in \mathbb N.
 \]
Hence,
\[
 \langle \tau_h \cdot \A u_h, \varphi \rangle = \langle f_h, u_h \varphi \rangle - \langle \tau_h, 
 \sum_{\alpha,\beta} c_{\alpha\beta} \partial^\alpha u_h \partial^\beta 
  \varphi \rangle.
\]
By the compactness assumption on $\A$ we may assume without loss of generality that $u_h \wc u$ in $\mathrm W^{k,2}(\Omega;\R^d)$.
Thus, passing to the limit we obtain
\begin{align*}
 \lim_{h \to \infty} \langle \tau_h \cdot \A u_h, \varphi \rangle  = \langle f,  u\varphi \rangle - \langle  \tau, 
 \sum_{\alpha,\beta} c_{\alpha\beta} \partial^\alpha  u \, \partial^\beta
  \varphi \rangle  = \langle  \tau \cdot \A  u, \varphi \rangle,
\end{align*}
for every $\varphi \in \mathcal D(\Omega)$.
One concludes that 
\begin{equation}\label{distributions}
\sigma_h \A u_h \cdot \A u_h \to \sigma \A u \cdot \A u \quad \text{in $\mathcal D'(\Omega)$}. 
\end{equation}
Fix $\omega \subset \subset \Omega$ and let $0 \le \varphi \in \mathcal D(\Omega)$ with
$\varphi \equiv 1$ on $\omega$. Using the convergence in (\ref{distributions}) and the uniform ellipticity \eqref{eq:ellipticity} of $\{\sigma_h\}$, one gets
\begin{align*}
 \lim_{h \to \infty} \| \A u_h - \A u\|_{\Lrm^2(\omega)} & \le  
 M \cdot \lim_{h \to \infty}  \langle  \sigma_h(\A(u_h - u))\cdot\A(u_h  - u), \varphi \rangle \\
 & \le M \cdot \bigg(\lim_{h \to \infty} \langle \sigma_h\A u_h \cdot \A u_h , 
 \varphi \rangle \\ 
 & - \lim_{h \to \infty} 2\langle \sigma_h \A u_h \cdot \A u, \varphi\rangle 
 + \langle \sigma_h \A u \cdot \A u , \varphi \rangle\bigg)\\
 &  = 0.
\end{align*}

\end{proof}

\subsection{Young measures and lower semi-continuity of integral energies}

In this section $\B : \Lrm^2(\Omega;Z) \to \Wrm^{-m,2}(\Omega;\R^n)$ is assumed to be a an $m$-th order homogeneous partial differential operator of the form
\[
\sum_\alpha B_\alpha \partial^\alpha, \qquad B_\alpha \in \text{Lin}(Z;\R^n), \;\text{with $Z$ a linear subspace of $\R^{dN^k}$}, 
\]  
satisfying the constant rank condition (\ref{rank}).

Next, we recall some facts about $\B$-quasiconvexity, lower semi-continuity and Young measures. The results in this section hold for differential operators with coefficients $B_\alpha$ in arbitrary spaces $\text{Lin}(\R^{p};\R^{q})$ for $p,q$ a pair of positive integers; however, we only present versions where the dimensions match our current setting. We start by stating a version of the Fundamental theorem for Young measures due to Ball \cite{Ball89}. 
\begin{theorem}[Fundamental theorem for Young measures]\label{fundamental} Let $\Omega \subset \R^N$ be a measurable set with finite measure and let $\{v_j\}$ be a 
sequence of measurable functions $v_j : \Omega \to Z$. Then there exists a subsequence $\{v_{h(j)}\}$ and a weak$^*$ measurable map $\mu : \Omega \to \M(Z)$ with the 
following properties:
\begin{enumerate}
 \item We denote $\mu_x := \mu(x)$ for simplicity, then $\mu_x \ge 0$ in the sense of measures and $|\mu_x|(Z) \le 1$ for a.e. $x \in \Omega$.\vskip2pt
 \item If one additionally assumes that $\{v_{h(j)}\}$ is uniformly bounded in $\Lrm^1(\Omega;Z)$, then $|\mu_x|(Z) = 1$ for a.e. $x \in \Omega$.\vskip2pt
 \item If $F : \R^{dN^k} \to \R$ is a Borel and lower semi-continuous function, and is also bounded from below, then 
 \[
  \int_\Omega \langle \mu_x, F\rangle \, \dd x \le \liminf_{j \to \infty} \int_\Omega F(v_{h(j)}) \, \dd x. 
 \]
 \item If $\{v_{h(j)}\}$ is uniformly bounded in $\Lrm^1(\Omega;Z)$ and $F: \R^{dN^k} \to \R$ is a continuous function, and bounded from below, then 
 \[
  \int_\Omega \langle \mu_x, F\rangle \, \dd x = \liminf_{j \to \infty} \int_\Omega F(v_{h(j)}) \, \dd x
 \]
if and only if $\{F\circ v_{h(j)}\}$ is equi-integrable. In this case, 
\[
 F \circ v_{h(j)} \wc \langle \mu_x,F\rangle \quad \text{in } \Lrm^1(\Omega).
\]
\end{enumerate}
\end{theorem}
In the sense of Theorem \ref{fundamental}, we say that the sequence $\{v_{h(j)}\}$ generates the Young measure $\mu$. 

The following proposition tells us that a uniformly bounded sequence in the $\Lrm^p$ norm, which is also sufficiently close to $\ker(\B)$, may be approximated by
a $p$-equi-integrable sequence in $\ker(\B)$ in a weaker $\Lrm^q$ norm. We remark that this rigidity result is the only one where Murat's constant rank condition (\ref{rank}) is 
used. 

\begin{proposition}[{\cite[Lemma 2.15]{FonsecaMuller99}}]\label{equi-p}
 Let $1 < p < \infty$. Let $\{v_h\}$ be a bounded sequence in $\Lrm^p(\Omega;Z)$ generating a Young measure $\mu$,
 with $v_h \wc v$ in $\Lrm^p(\Omega;Z)$ and $\B v_h \to 0$ in $\mathrm W^{-m,p}(\Omega;\R^n)$. Then there exists a $p$-equi-integrable sequence $\{u_h\}$ in $\Lrm^p(\Omega;Z) \cap \ker(\B)$  that generates the same Young measure $\mu$
 and is such that
 \[
  \int_\Omega v_h \, \dd x = \int_\Omega u_h \, \dd x, \quad \|v_h - u_h \|_{\Lrm^q(\Omega)} \to 0, \quad \text{for all } 1 \le q < p. 
 \]
\end{proposition}

Let $F : \R^{dN^k} \to \R$ be a lower semi-continuous function  with $0 \le F(P) \le C(1 + |P|^p)$ for some positive constant $C$. 
The $\B$-quasiconvex envelope of $F$ at $P  \in  Z \subset \R^{dN^k}$ is defined as
\begin{equation}\label{eq:integral}
\begin{split}
 Q_{\B} F(P) := & \inf\setBB{\int_{[0,1]^N} F(P + v(y)) \dd y}{\\ & v \in \Crm_{\text{per}}^\infty\left([0,1]^N;Z\right), \B v = 0 \;\, \text{and}  \; \int_{[0,1]^N} v \dd y = 0}.
 \end{split}
\end{equation}
The most relevant feature of $Q_{\B}F$ is that, for $p > 1$, the lower semi-continuous envelope with respect to the weak-$L^p$ topology of the functional 
\begin{equation}\label{eq:integralq}
 v \mapsto  \int_\Omega F(v) \, \text{d}x, \qquad \text{where $v \in \Lrm^p(\Omega;Z)$ and $\B v = 0$},
\end{equation}
is given by the functional
 \[v \mapsto \int_\Omega Q_{\B}F(v) \, \text{d}x, \qquad \text{where $v \in \Lrm^p(\Omega;Z)$ and $\B v = 0$}.\]
 
 If $\mu$ is a Young measure generated by a 
 sequence $\{v_h\}$ in $\Lrm^p(\Omega;Z)$ such that $\B v_h = 0$ for every $h \in \mathbb N$, then we say that $\mu$ is a $\B$-free Young measure.
 
  We recall the following Jensen inequality for $\B$-free Young measures \cite[Theorem 4.1]{FonsecaMuller99}:
 \begin{theorem}\label{characterization}
  Let $1 < p < +\infty$. Let $\mu$ be a $\B$-free Young measure in $\Omega$. Then for a.e. $x \in \Omega$ and all lower semi-continuous functions that
  satisfy $|F(P)| \le C(1 + |P|^p)$ for some positive constant $C$ and all $P \in \R^{dN^k}$, one has that
  \[
   \langle \mu_x , F \rangle \ge Q_{\B}F(\langle \mu_x, \Id\rangle).
  \]
 \end{theorem}
 
 \subsection{Geometric measure theory and sets of finite perimeter}\label{sec:gmt}
 
 Most of the facts collected in this section can be found in \cite{MaggiBook} and \cite{APF}; however, some notions as the slicing of sets of finite perimeter are presented there only in a formal way. For a better understanding of such topics we refer the reader to \cite{FedererBook}. 
 
 Let $A \subset \R^N$ be a Borel set. The Gauss-Green measure $\mu_A$ of $A$ is the derivative of the characteristic function of $A$ in the sense of distributions, i.e., 
$\mu_A := \nabla(\mathds 1_A)$. 
       We say that $A$ is a set of locally finite perimeter if and only if $|\mu_A|$ is a vector-valued Radon 
       measure in $\R^N$. We write $A \in \BV_{\text{loc}}(\R^N)$ to express that $A$ is a set of locally finite perimeter in $\R^N$.
       
       Let $\omega \subset \subset \R^N$ be a Borel set. The perimeter in $\omega$ of a set $A$ with locally finite perimeter is defined as
       \[
        \Per(A,\omega) := |\mu_A|(\omega).
       \]

The Radon-Nikod\'ym differentiation theorem states that the set of points
\begin{align*}
 \partial^* A   := \setBB{x \in \R^N}{& \quad \lim_{r  \downarrow 0} \; \frac{\Per(A;B_r(x))}{\text{vol}(B_1')\cdot r^{N-1}}  = 1,  \\ &  \text{and} \quad  \frac{\textnormal{d}\mu_A}{\textnormal{d}|\mu_A|}(x) \; \text{exists and belongs to $\mathbb S^{N-1}$}}
\end{align*}
has full $|\mu_A|$-measure in $\R^N$; this set is commonly known as the {\it reduced boundary} of $A$.
We will also use the notation 
\[
 \nu_A(x) := \frac{\textnormal{d}\mu_A}{\textnormal{d}|\mu_A|}(x) \qquad x \in \partial^* A;
\]
the {\it measure theoretic normal} of $A$.

In general, for $s \ge 0$, we will denote by $\mathcal H^s$ the {\it $s$-dimensional Hausdorff measure in $\R^N$}.
The following well-known theorem captures the structure of sets with finite perimeter in terms of the measure $\mathcal H^{N-1}$:
 \begin{theorem}[De Giorgi's Structure Theorem]
            Let $A$ be a set of locally finite perimeter. Then
            \begin{gather*}
            \partial^* A = \bigcup_{j = 1}^\infty K_j \cup N,
            \end{gather*}
            where 
            \[|\mu_A|(N) = 0,\]
            and $K_j$ is a compact subset of a $\Crm^1$-hypersurface $S_j$ for every $j \in \mathbb N$.
Furthermore, $\nu_A|_{S_j}$ is normal to $S_j$ and 
\[
 \mu_A = \nu_A\,\mathcal H^{N-1}\llcorner \partial ^* A.\]
\end{theorem}
From De Giorgi's Structure Theorem it is clear that $\text{spt}~ \mu_A = \overline {\partial^* A}$. Actually, up to modifying $A$ on a set of zero measure, one has that $\partial A = \overline {\partial^* A}$ (see \cite[Proposition 12.19]{MaggiBook}). 
From this point on, each time we deal with a set $A$ of finite perimeter, we will assume without loss of generality that
\begin{equation}\label{eq:closed}
\partial A = \supp ~ \mu_A = \overline {\partial^* A}.
\end{equation}

For a set of locally finite perimeter $A$, the deviation from being a {\it perimeter minimizer} in $\Omega$ at a given scale $r$ is quantified by the monotone function
\[
\Dev_\Omega(A,r) := \sup \bigg\{ \Per(A;B_r(x)) - \Per(E;B_r(x)) : E \Delta A \subset \subset B_r(x) \subset \Omega\bigg\}.
\] 
The next result, due to Tamanini \cite{Tamanini84}, states that a set of locally finite perimeter with small deviation $\Dev_\Omega$ at every scale is actually a $\Crm^1$-hypersurface up to a lower dimensional set.  
\begin{theorem}\label{tamanini}
  Let $A \subset \R^N$ be a set of locally finite perimeter and let $c(x)$ be a locally bounded function for which
  \[\Dev_{\Omega}(A,r) \le c(x)r^{N-1 + 2\eta} \qquad \text{for some} \; \eta \in (0,1/2\,].\]
 Then the reduced boundary in $\Omega$, $(\partial^* A \cap \Omega)$, is an open $\Crm^{1,\eta}$-hypersurface and the singular set 
 $\Omega \cap (\partial A \setminus \partial^* A)$ has at most 
 Hausdorff dimension $(N-8)$. 
 \end{theorem}
 
%

\subsubsection{Slicing sets of finite perimeter} Given a Borel set $E \subset \R^N$ and a Lipschitz function $g: \R^N \to \R$, we shall consider the level set slices 
\[
E_t \coloneqq E \cap \big\{ g = t \big\}, \quad t \in \R. 
\]

For a set $A\subset \R^N$ of finite perimeter in $\Omega$, the level set slice of the reduced boundary $(\partial^* A)_t$ is $\mathcal H^{N-2}$-rectifiable for almost every $t \in \R$. Furthermore, by the co-area formula,
$t \mapsto \mathcal H^{N-2}( (\partial ^*A)_t) \in \Lrm^1_{\text{loc}}(\R)$.

If the set $\{g = t\}$ is a $\Crm^1$-manifold and $t$ is such that $\mathcal H^{N-2}( (\partial ^*A)_t) < \infty$, we shall define the {\it slice of $A$ in $g^{-1}\{t\}$} as
\[
\langle A, g, t \rangle \coloneqq 
 \mathcal H^{N-2} \llcorner (\partial^* A)_t.
\]

It turns out that, for $g(x) = |x|$, the level set slice $A_t$ is locally diffeomorphic to a set of finite perimeter in $\R^{N-1}$. Even more, 
\begin{gather}\label{eq:agree}
\mathcal H^{N-2} \llcorner \partial^* A_t = \langle A, g, t\rangle \quad \text{for a.e. $t > 0$, and} \\
\pi_g \nu_A \coloneqq  (\Id_{\R^N} - \nabla g \otimes \nabla g) \nu_A \neq 0 \quad \text{for $\mathcal H^{N-2}$-a.e. $x \in (\partial^*A)_t$}. \label{eq:agree2}
\end{gather}
Here, $\partial^*A_t$ is understood as the image, under local diffeomorphisms, of the reduced boundary of a set of finite perimeter. These properties can be inferred from the classical slicing by hyperplanes, see e.g., \cite[Chapter 18.3]{MaggiBook}.

We also define the cone extension of a set $E \subset \R$ containing $\{0\}$ by letting
\[
D_E \coloneqq \bigg\{ \lambda x \in \R^N : \lambda > 0, \; x \in E \bigg\}.
\] 
For a.e. $t > 0$ and $g(x) = |x|$, the cone extension of $A_t$ is a set of  locally finite perimeter 
 in $\R^N$ with 
 	\begin{equation}\label{eq:conegrowth}\partial^* D_{A_t} = D_{(\partial^* A)_t} \quad \text{and} \quad 
 \Per(D_{A_t};B_\rho) = \left(\frac{1}{N-1}\right) \frac{\rho^{N-1}}{t^{N-2}} \cdot \mathcal H^{N-2}((\partial^* A)_t).
 \end{equation}

 In order to attend different variational problems involving the minimization of perimeter, a well-known technique is to modify a set $A$ within balls $B_t$ without modifying its Gauss-Green measure in $(B_t)^c$. 
 
 For almost every $t > 0$, where $\langle A, g, t \rangle$ is well-defined and \eqref{eq:agree}-\eqref{eq:agree2} hold, we construct a cone-like comparison set of $A$ by setting
\begin{equation}\label{eq:cone}
\tilde A \coloneqq  \mathds{1}_{B_t}D_{A_t} + \mathds{1}_{\Omega \setminus B_t}A.
\end{equation}
Exploiting the basic properties of reduced boundaries, it follows by \eqref{eq:agree} that
\begin{equation}\label{eq:conelike}
\mu_{\tilde A} =  \mu_{D_{A_t}} \llcorner B_t + \mu_A \llcorner (B_t)^c;
\end{equation}
and, in particular,
\[
\Per(\tilde A;B_r) = \Per(D_{\partial^* A_t};B_t) + \Per(A;(B_t)^c \cap B_r) \qquad \text{for all $r > t$}.
\]

On the other hand, again by the co-area formula,
\[
\mathcal H^{N-1}((\partial^* A)_t \cap \{ g= t\}) = 0 \qquad \text{for almost every $t > 0$}.
\]
Using the monotonicity of $r \mapsto \Per(A;B_r)$ and the general version of the co-area formula (see \cite[Theorem  3.2.22]{FedererBook}) one can show that the derivative of $r \mapsto \Per(A;B_r)$ exists at almost every $t > 0$; even more, up to a further null set it is given by
\begin{equation}\label{eq:slice}
\frac{\dd }{\dd r} \bigg|_{r = t} \Per(A;B_r)  = |\pi_t \nu_A|^{-1} \mathcal H^{N-2}((\partial^* A)_t)\ge 
\langle A, g , t \rangle(\R^N).
\end{equation}
The previous estimate will play a crucial role in proving the lower bound \eqref{lower}.

\section{Existence of solutions: proof of Theorem \ref{existence}}\label{exsol}

We show an equivalence between the constrained problem (\ref{intro1}) and the unconstrained problem \eqref{P} for which
existence of solutions and regularity properties for minimizers are discussed in the present and subsequent sections. 
We fix $\A: \Lrm^2(\Omega\;\R^d) \to \Wrm^{-k,2}(\Omega;\R^{dN^k})$ an operator of gradient from as in Definition \ref{gto}. We also fix $A_0 \subset \R^N$, a set of locally finite perimeter. 

Recall that, the minimization problem (\ref{intro1}) under the state constraint (\ref{state eq}) reads:
\begin{equation*}
\text{minimize} \qquad \left\{\; \int_\Omega fw_A \; + \; \Per(A;\overline\Omega)\;  :\; A \in \BV_{\text{loc}}(\R^N), ~A\cap \Omega ^c \equiv A_0\cap \Omega^c \right\},\footnote{As stated in Section \ref{sec:gmt}, we write $A \in \BV_{\text{loc}}(\R^N)$ to express that $A$ is a Borel set of locally finite perimeter in $\R^N$.}
\end{equation*}
where $w_A$ is the unique distributional solution to the state equation
\[
\A^*(\sigma_A \A u) = f, \quad u \in \mathrm W^{\A}_0(\Omega).
\]
On the other hand, the associated saddle point problem \eqref{P} reads:
\[
\inf\left\{ \sup_{u \in \Wrm^{\A}_0(\Omega)} I_\Omega(u,A) : A \in \BV_{\text{loc}}(\R^N), \; A \cap \Omega^c \equiv A_0 \cap \Omega^c\right\}, \tag{P}
\]
where 
\[
I_\Omega(u,A) \coloneqq \int_\Omega 2fu \, \text{d}x - \int_\Omega \sigma_A \A u \cdot \A u \, \text{d}x ~ + ~ \Per(A;\overline \Omega).
\]
\begingroup
\def\thetheorem{\ref{existence}}
\begin{theorem}[existence]
There exists a solution $(w,A)$ of 
problem \eqref{P}. Furthermore, there is a one to one correspondence 
\[
 (w,A) \mapsto (w_A,A) 
\]
between solutions to problem \eqref{P} and the minimization problem
(\ref{intro1}) under the constraint (\ref{state eq}).
\end{theorem}
\addtocounter{theorem}{-1}
\endgroup
\begin{proof} We employ the direct method. We begin by proving existence of solutions to problem \eqref{P}. To do so, we will first prove the following:  \begin{claim}\label{claime}{\it 1}.  For any set $A \subset \R^N$ as in the assumptions, there exists $w_A \in \mathrm W^{\A}_0(\Omega)$ such that 
\[0 \le I_\Omega(w_A,A) = \sup_{u \in \mathrm W^{\A}_0(\Omega)} I_\Omega(u,A) <  \infty.\]
\end{claim}
The tensor $\sigma_A$ is a positive definite tensor and therefore the mapping
\[
 u \mapsto I_\Omega(u,A) = \int_\Omega 2fu - \sigma_A \A u \cdot \A u \, \text{d}x + \text{Per}(A;\overline\Omega)
\]
is strictly concave. Observe that $\sup_{u \in \mathrm W^{\A}_0(\Omega)} I_\Omega(u,A) \ge \text{Per}(A;\overline\Omega)$; indeed, we may take $u \equiv 0 \in \mathrm W^{\A}_0(\Omega)$. Hence,
\begin{equation}\label{positive}
 \sup_{u \in \Wrm^{\A}_0(\Omega)} I_\Omega(u,A) \ge \text{Per}(A;\overline\Omega)\ge 0.
\end{equation}
Because of this, we may find a maximizing sequence $\{w_h\}$ in $\mathrm W^{\A}_0(\Omega)$, i.e., 
\[
 I_\Omega(w_h,A) \to \sup_{u \in \mathrm W^{\A}_0(\Omega)} I_\Omega(u,A), \quad \text{as $h$ tends to infinity}.
\]
Even more, one has from \eqref{eq:ellipticity} that
\[
 -\frac{1}{M}\|\A w_h\|^2_{\Lrm^2(\Omega)} \ge -\int_\Omega \sigma_A \A w_h \cdot \A w_h \, \text{d}x 
\]
and consequently from \eqref{positive} and \eqref{poincare2} one infers that
\begin{equation}\label{muere}
C(\Omega)^{-1} \cdot \limsup_{h \to \infty}  \frac{1}{M}\|w_h\|^2_{\Lrm^2(\Omega)} \le  \limsup_{h \to \infty} \frac{1}{M}\|\A w_h\|^2_{\Lrm^2(\Omega)} \le 2\|f\|_{\Lrm^2(\Omega)}\cdot\limsup_{h \to \infty}  \|w_h\|_{\Lrm^2(\Omega)}.
\end{equation}
A fast calculation shows that $\|w_h\|_{\Lrm^2(\Omega)} \le 2MC(\Omega)\|f\|_{\Lrm^2(\Omega)}$;
in return, (\ref{muere}) also implies that
\[
  \limsup_{h \to \infty} \|\A w_h\|^2_{\Lrm^2(\Omega)} \le 4C(\Omega)M^2\|f\|^2_{\Lrm^2(\Omega)}.
\]
Hence, using again the compactness property of $\A$, we may pass to a subsequence (which we will not relabel) and find 
$w_A \in \mathrm W^{\A}_0(\Omega)$ with
\[
 w_h \to w_A \quad \text{in } \Lrm^2(\Omega;\R^d), \qquad \A w_h \wc \A w_A \quad \text{in }  \Lrm^2(\Omega;\R^{dN^k}).
\]
The concavity of $-\sigma_A z \cdot z$ is a well-known sufficient condition for the upper semi-continuity of the functional $\A u \mapsto -\int_\Omega \sigma_A \A u \cdot \A u$.  Therefore, 
\[
 \sup_{u \in \mathrm W^{\A}_0(\Omega)} I_\Omega(u,A) = \lim_{h \to \infty} I_\Omega(w_h,A) \le I_\Omega(w_A,A). 
\]This proves the claim. 

Now, we use {\it Claim 1} to find a minimizing sequence $\{A_h\}$ for $A \mapsto I_\Omega(w_{A},A)$.
Since the uniform bound (\ref{muere}) does not depend on $A$, we may again assume (up to a subsequence) that there exists $\tilde w \in \mathrm W^{\A}_0(\Omega)$ such that
\[
 w_{A_h} \to \tilde w \quad \text{in } \Lrm^2(\Omega;\R^d), 
 \qquad \A w_{A_h} \wc \A \tilde w \quad \text{in } \Lrm^2(\Omega;\R^{dN^k}), \quad \text{and} \quad \A^* (\sigma_{A_h }\A w_{A_h}) = f.
\]
Even more, since $\{A_h\}$ is minimizing, it must be that $\sup_h \{\Per(A_h;B_R)\} < \infty$, for some ball $B_R$ properly containing $\Omega$, and thus 
(for a further subsequence) there exists a set  $\tilde A \subset \R^N$ of locally finite perimeter with $\tilde A \cap \Omega^c \equiv A_0 \cap \Omega^c$ and such that
\[
 \mathds 1_{A_h} \to \mathds 1_{\tilde A} ~\text{in }  \Lrm^1(B_R), \qquad |\mu_{\tilde A}|(B_R) \le \liminf_{h \to \infty} \; |\mu_{\tilde A_h}|(B_R).
\]
Therefore
\begin{equation}\label{maña}
\begin{split}
\Per(\tilde A;\overline \Omega &) =
 |\mu_{\tilde A}|(B_R) - |\mu_{A_0}|(B_R \setminus \overline \Omega) \\ 
 & \le  \liminf_{h \to \infty} |\mu_{A_h}|(B_R) - |\mu_{A_0}|(B_R \setminus \overline \Omega) = \liminf_{h \to \infty}  \;\Per(A_h;\overline \Omega)
 \end{split}
\end{equation}
A consequence of Lemma \ref{compensated} is that
\begin{equation}\label{suff}
\A^* (\sigma_{\tilde A}\A \tilde w) = f \quad \text{ in $\mathcal D'(\Omega;\R^d)$, \quad and \quad $\int_\Omega \sigma_{A_h} \A w_{A_h} \cdot \A w_{A_h} \to \int_{\Omega}\sigma_{\tilde A} \A \tilde w \cdot \A \tilde w$}.\footnote{The convergence of the total energy is not covered by Lemma \ref{compensated}; however, this can be deduced using integration by parts and the fact that $w_h$ has zero boundary values for every $h \in \mathbb N$.}
\end{equation}
By taking the limit as $h$ goes to infinity we get from \eqref{maña} and the convergence above that
\[
 \min_A \; \sup_{u \in \Wrm^{\A}_0(\Omega)} \; I_\Omega (u,A) = \lim_{h \to \infty} ~ I_{\Omega}(w_{A_h},A_h) \ge
 I_{\Omega}(\tilde w,\tilde A) = I_\Omega(w_{\tilde A},\tilde A),
\] 
where the last equality is a consequence of the identity $\tilde w = w_{\tilde A}$ which can be easily derived by using the equation and the strict concavity of $I_\Omega$ in the first variable. 
Thus, the pair $(w_{\tilde A}, \tilde A)$ is a solution to problem \eqref{P}.

The equivalence of problem \eqref{P} and problem (\ref{intro1}) under the state constraint (\ref{state eq}) follows easily from (\ref{suff}), the strict concavity of $I_\Omega(\cdot,A)$, and a simple integration by parts argument. 
 \end{proof}

\section{The energy bound: proof of Theorem \ref{thm2}}\label{CE}

Throughout this section and for the rest of the manuscript we 
fix $\A : \Lrm^2(\Omega;\R^d)\to \Wrm^{-k,2}(\Omega;\R^{dN^k})$ in the class of operators of gradient form. Accordingly,
the notations $Z$ and $\B$ shall denote the subspace of
$\R^{dN^k}$ and the homogeneous operator associated to $\A$  (see Definition \ref{gto}).
We will also write $(w,A)$ to denote a particular solution of problem \eqref{P}. 

Consider the energy $J_\omega : \Lrm^2(\Omega;Z) \times \mathfrak B(\Omega) \to \R$ defined as
\[
J_{\omega}(v,E) \coloneqq \int_{\omega} \sigma_E v \cdot v \dd y \; + \; \Per(E;\omega), \qquad \text{for $\omega \subset \Omega$ an open set}. 
\]

The goal of this section is to prove a local bound for the map $x \mapsto J_{B_r(x)}(\A w,A)$. More precisely, we aim to prove
that for every compactly contained subset $K$ of $\Omega$ there exists a positive number $\Lambda_K$ such that
\begin{equation}\label{eq:upper}
 J_{B_r(x)}(\A w,A) \le \Lambda_K r^{N-1} \qquad \text{for all $x \in K$ and every $r \in (0,\dist(K,\partial \Omega))$}. 
\end{equation}
 Our strategy will be the following.  We first define a one-parameter family $J^\varepsilon$ of perturbations of $J_{B_1}$ in the perimeter term. In Theorem \ref{gamma} we show that, as the perimeter term
 vanishes, these perturbations $\Gamma$-converge (with respect to the $\Lrm^2$-weak topology) to the relaxation of the energy
 \[
  w \mapsto  \int_\Omega W(\A w) \, \text{d}x,
 \]
for which we will assume certain regularity properties (cf. property (\ref{eq:assumption})). Then, using a compensated compactness argument, we prove Theorem \ref{thm2} (upper bound) by transferring the regularity properties of the relaxed problem to our original problem.

Before moving forward, let us shortly discuss how the higher integrability property (\ref{eq:assumption}) stands next to the standard assumption that the materials
$\sigma_1$ and $\sigma_2$ are well-ordered.

\subsection{A digression on the regularization assumption}

As commented beforehand in the introduction, a key assumption in the proof of the upper bound \eqref{eq:upper} is that {\it generalized} local minimizers of the energy
\[
 u \mapsto \int_{B_1} W(\A u) \, \text{d}y, \qquad \text{where $u \in \Wrm^{\A}(B_1)$},
\]
possess improved decay estimates. More precisely, we require that {\it local}  minimizers $\tilde u$ of the functional
\begin{equation}\label{dig}
 u \mapsto \int_{B_1} Q_{\B}W(\A u)\, \text{d}y, \qquad \text{where $u \in \Wrm^{\A}(B_1)$},
\end{equation}
possess a higher integrability estimate of the form 
\begin{equation*}\label{eq:campanato}
[\A \tilde u]^2_{\Lrm^{2,N-\delta}(B_{1/2})} \le c\| \A \tilde u\|^2_{\Lrm^2(B_1)} \quad \text{for some $\delta \in [0,1)$}. 	\tag{Reg}
\end{equation*}
Only then, we will be able to 
transfer a decay estimate of order $\rho^{N-1}$ to solutions of our original problem. 

\begin{remark}[the case of gradients] In the case $\A = \nabla$, condition \eqref{eq:campanato} boils down to regularity above the critical $\Crm^{0,1/2}$ local regularity. More specifically, 
	\[
\frac{1}{r^{N- \delta + 2}}	\int_{B_r(x)} |w - (w)_{r,x}|^2 \,\text{d}y \le [\nabla w]^2_{\Lrm^{2,\,N-\delta}(B_{1/2})} \le c\| \nabla w\|^2_{\Lrm^2(B_1)} \qquad \text{for all $B_r(x) \subset B_{1/2}$}.
	\]
	By Poincar\'e's inequality and Campanato's Theorem one can easily deduce (cf. \cite{LinKohn99}) that \[w \in \Crm^{0,\frac{1}{2} + \varepsilon}_{\text{\textnormal{loc}}}(B_{1/2}).\]
	\end{remark}
	
	Let us give a short account of some cases where one may find \eqref{eq:assumption} to be a {\it natural assumption}.

\subsubsection{The well-ordered case} The notion of well-ordering in Materials Science is not only justified as the comparability of two materials, 
one being at least {\it better} than the other. It has also been a consistent assumption when dealing with optimization problems because it allows 
explicit calculations. See for example \cite{Hash63,AllaireKohn93,AllaireKohn94}, where the authors discuss how the well-ordering assumption plays a role in proving
the optimal lower bounds of an effective tensor made-up by two materials. 
If $\sigma_1$ and $\sigma_2$ are well-ordered, say $\sigma_2 \ge \sigma_1$ as quadratic forms, then $W(P) = \sigma_2 P \cdot P$.
Hence, by Lemma \ref{lem:cosntant}, the desired higher integrability \eqref{eq:assumption} holds with $\delta = 0$.

\subsubsection{The non-ordered case} Applications for this setting are mostly reserved for the {\it scalar case}. In this particular case  one can ensure that $Q_{\B}W = W^{**}$, where $W^{**}$ is the convex
envelope of $W$. For example, one may consider an optimal design problem involving the linear conductivity equations for two dielectric materials which happen to be incomparable 
as quadratic forms. In this setting, it is not hard to see that indeed $QW = W^{**}$ and even that $W^{**} \in \Crm^{1,1}(\R^{dN^k},\R)$. In dimensions $N =2,3$, one can employ a 
Moser-iteration technique for the dual problem as the one developed in \cite{CarMull02} to show better regularity of minimizers of (\ref{dig}).

 Regarding the case of systems, if no well-ordering of the materials is assumed,
it is not clear to us that \eqref{eq:assumption} necessary holds (compare to \cite{Evans86,SvYan02}). 

\subsection{Proof of Theorem \ref{thm2}}

We define an $\varepsilon$-perturbation of $v \mapsto \int_{B_1} \sigma_A v \cdot v$ as  follows. Consider the 
functional 
\begin{equation}\label{I}
 (v,A) \mapsto J^\varepsilon(v,A) \coloneqq \int_{B_1} \sigma_A v \cdot v \, \text{d}y ~ + ~  \varepsilon^2 \Per(A;B_1), \quad \text{for $\epsilon \in [0,1]$}; \qquad J \coloneqq J^1.
\end{equation}
 By a scaling argument one can easily check that
 \begin{equation}\label{scaling}
  \varepsilon^2 J(v,A) = J^\varepsilon (\varepsilon v,A).
 \end{equation}
Furthermore, 
\begin{equation}\label{scaling2}
\text{$v$ is a local minimizer of $J(\,\cdot\,,A)$ if and only if $\varepsilon v$ is a local minimizer of 
$J^\varepsilon(\,\cdot\,,A)$}. 
\end{equation}
We also consider the following one-parameter family of functionals:
 \begin{equation}\label{J}
  v \mapsto G^\varepsilon(v) \coloneqq \begin{cases}
                         \displaystyle{\min_{A \in \mathfrak B(B_1)}} ~  J^\varepsilon(v,A) &\quad \text{if \quad $v \in \Lrm^2(\Omega;Z)$ and $\B v = 0$},\vspace{2pt}
  \\
  \infty & \hfill \text{otherwise}.
                        \end{cases}
\end{equation}
The next result characterizes the $\Gamma$-limit of these functionals as $\varepsilon$ tends to zero.

\begin{theorem}\label{gamma} The $\Gamma$-limit of the functionals $G^\varepsilon$, as $\varepsilon$ tends to zero, and with respect to the weak-$\Lrm^2$ topology
is given by the functional 
\begin{equation}
 G(v) \coloneqq \begin{cases}
            \displaystyle \int_{B_1}Q_{\mathcal B}W(v) \, \textnormal{d}y 
            & \qquad \text{if \quad $v \in \Lrm^2(\Omega;Z)$ and $\B v = 0$},\vspace{2pt}
            \\
            \displaystyle \infty & \hfill \text{else}.
            \end{cases}
\end{equation}
\end{theorem}

\begin{proof} We divide the proof into three steps. First, we will prove the following auxiliary lemma.

\begin{lemma}\label{equi}
 Let $\omega \subset \R^N$ be an open and bounded domain. Let $p > 1$ and let $F:\R^{dN^k} \to [0,\infty)$ be a continuous integrand with $p$-growth, i.e., 
 \[0 \le F(P) \le C(1 + |P|^p), \qquad P \in \R^{dN^k}.\] If $v \in \Lrm^p(\omega;Z)$ and $\B v = 0$, then
 there exists a $p$-equi-integrable recovery sequence $\{v_h\} \subset \Lrm^p(\omega;Z)$ for $v$ 
 such that
 \[
  \B v_h = 0 \qquad \text{and} \qquad F(v_h) \wc Q_{\mathcal B}F(v) \quad \text{in $\Lrm^1(\omega)$}.
 \] 
\end{lemma}
\begin{proof}
 Since $v \mapsto \int_\omega Q_{\B} F(v)$ is the lower semi-continuous envelope of $v \mapsto \int_{\omega} F(v)$ (see \eqref{eq:integral}-\eqref{eq:integralq}) with respect to the weak-$\Lrm^p$ topology, 
we may find a sequence $\{v_h\}$ with the following properties:
\[ \B v_h = 0,  \qquad v_h \stackrel{\Lrm^p}\wc v,\]
and
\[
 \int_{\omega} Q_{\mathcal B}F(v) \, \text{d}x \ge \int_{\omega} F(v_h) \, \text{d}x - \frac{1}{h}.
\]
Passing to a subsequence if necessary, we may assume that the sequence $\{v_h\}$ generates a $\B$-free Young measure which we denote by $\mu$. We then apply 
\cite[Lemma 2.15]{FonsecaMuller99} to find a $p$-equi-integrable sequence $\{v'_h\}$ (with $\B v_h = 0$) generating
the same Young measure $\mu$. On the one hand, the
Fundamental Theorem for Young measures (Theorem \ref{fundamental}) and the fact that $\{v_h\}$ generates $\mu$ yield
\[
 \liminf_{h \to \infty} \int_{\omega} F(v_h) \, \text{d}x \ge \int_{\omega} \langle \mu_x,F\rangle \, \text{d}x.
\]
On the other hand, due to the same theorem and the equi-integrability of the sequence $\{|v'_h|^p\}$
one gets the convergence $F(v'_h) \wc \langle \mu_x,F\rangle \in \Lrm^1$. In other words, 
\[
  \lim_{h \to \infty}  \int_{\omega} F(v'_h) \, \text{d}x = \int_{\omega} \langle \mu_x,F\rangle \, \text{d}x.
\]
The three relations above yield 
\begin{equation}\label{perro}
 \int_{\omega} Q_{\mathcal B}F(v) \, \text{d}x \ge \limsup_{h \to \infty}  \int_{\omega} F(v_h) \ge 
 \int_{\omega} \langle \mu_x,F\rangle \, \text{d}x = \lim_{h \to \infty}  \int_{\omega} F(v'_h) \, \text{d}x
 \ge  \int_{\omega} Q_{\mathcal B}F(v) \, \text{d}x.
\end{equation}
We summon the characterization for $\B$-free Young measures from Theorem \ref{characterization} to observe that 
\[ \langle \mu_x,F\rangle \ge Q_{\B} F(\langle \mu_x, \Id \rangle) = Q_{\B} F(v(x)) \quad \text{a.e.} \; x \in \omega.
\]
This inequality and
(\ref{perro}) imply
\[
 \langle \mu_x,F\rangle = Q_{\B}F(v(x)) \quad \text{a.e.} \; x \in \omega.
\]
We conclude by recalling that $F(v'_h) \wc \langle \mu_x,F\rangle$ in $\Lrm^1(\omega)$.
 \end{proof}

\noindent {\bf The lower bound.} Let $v \in \Lrm^2(B_1;Z)$ and let $\{v_\varepsilon\}$ be a sequence in $\Lrm^2(B_1;Z)$ such that $v_\varepsilon \wc v$ in $\Lrm^2(B_1;Z)$.
 We want to prove that
 \[
  \liminf_{\varepsilon \downarrow 0} G^\varepsilon(v_\varepsilon) \ge G(v).
 \]
 Notice that, we may reduce the proof to the case where $\B v_\varepsilon = 0$ for every $\varepsilon$.
%
From the inequality $\sigma_A \ge W \ge Q_{\mathcal B}W$ (as quadratic forms), we infer that
\[
 J^\varepsilon(v_\varepsilon) \ge  
 \int_{B_1} Q_{\mathcal B}W(v_\varepsilon) \, \text{d}y.
\]
Next, we recall that $v \mapsto \int_{B_1} Q_{\B}W(v)$ is lower semi-continuous in $\{v \in \Lrm^2(\Omega;Z) : \B v = 0 \}$ with respect to the weak-$\Lrm^2$ topology.
Hence,
\[
\liminf_{\varepsilon \downarrow 0} G^\varepsilon(v_\varepsilon) \ge 
 \int_{B_1} Q_{\mathcal B}W(v) \, \text{d}y.
\]
This proves the lower bound inequality. \\

\noindent {\bf The upper bound.} We fix $v \in \Lrm^2(B_1;Z)$, we want to show that there exists a sequence
$\{v_\varepsilon\}$ in $\Lrm^2(B_1;Z)$ with $v_\varepsilon \wc v$ in $\Lrm^2(B_1;Z)$ and such that 
\[
\limsup_{{\varepsilon \downarrow 0}} G^\varepsilon(v_\varepsilon) \le G(v).
\]
We may assume that $\B v = 0$, for otherwise the inequality occurs trivially. 
Lemma \ref{equi} guarantees the existence of a
$2$-equi-integrable 
sequence $\{v_h\}_{h = 1}^\infty$ for which
\begin{equation}\label{4a}
\B v_h = 0, \quad  v_h\wc v \quad \text{in $\Lrm^2(B_1;Z)$}, \quad \text{and} \quad W(v_h) \wc Q_{\mathcal B}W(v) \quad 
 \text{in $\Lrm^1(B_1)$}.
\end{equation}
Next, we define an $h$-parametrized sequence of subsets of $B_1$ in the following way:
$$A_h \coloneqq \bigg\{x \in B_1 : (\sigma_1 - \sigma_2)v_h\cdot v_h \le 0\bigg\}. 
$$
Using the fact that smooth sets are dense in the broader class of subsets with respect to measure convergence, 
we may take a smooth set $A'_h \subset  B_1$ such that
the following estimates hold for some strictly monotone function $L: \mathbb N  \to \mathbb N$ (with $\lim_{h \to \infty} L(h) = \infty$):
\begin{equation}\label{6a}
|(A'_h \Delta A_h)\cap B_1| = O(h^{-1}), \qquad 
\Per(A'_h;B_1) \le L(h).
\end{equation}
Observe that, by the $2$-equi-integrability of $\{ v_h \}$, one gets that
\begin{equation}\label{6b}
\|(\sigma_{A_h} - \sigma_{A'_h}) v_h\cdot v_h\|_{\Lrm^2(B_1)} \le 
M\|v_h\|^2_{\Lrm^2(S_h)} = O(h^{-1}), \qquad \text{where $S_h \coloneqq A'_h \Delta A_h$}.
\end{equation}

The next step relies, essentially, on stretching the sequence $\{v_h \}$.
Define the $\varepsilon$-sequence 
\[
\overline v_\varepsilon \coloneqq v_{K(\varepsilon)}, \quad \varepsilon \le \frac{1}{L(1)},
\]
where $K: \R_+ \to \mathbb N$ is the piecewise constant decreasing function defined as
\[
K \coloneqq \sum_{h = 1}^\infty h \cdot \mathds 1_{R_h}, \quad R_h \coloneqq \left(\frac{1}{L(h +1)},\frac{1}{L(h)}\right].
\]
\begin{claim}\
\begin{itemize}
\item[1.] $L \circ K (\varepsilon) \le \varepsilon^{-1}$, if $\varepsilon \in (0,L(1)^{-1}]$.
\item[2.] $K (\varepsilon) = h$, where $h$ is such that $\varepsilon \in R_h$. 
\end{itemize}
\end{claim}
\begin{proof}
To prove (1), observe from the strict monotonicity of  $L$ that $\cup_{h = 1}^\infty R_h = (0, L(1)^{-1}]$. A simple calculation gives
\begin{align}
L(K(\varepsilon)) & = L(\sum_{h = 1}^\infty h\cdot \mathds 1_{R_h}(\varepsilon)) = \sum_{h= 1}^\infty L(h) \cdot \mathds 1_{R_h}(\varepsilon) = L(h_0) \cdot \mathds 1_{R_{h_0}}(\varepsilon) \le \frac{1}{\varepsilon},
\end{align}
where $h_0$ is such that $\varepsilon \in R_{h_0}$.
The proof of (2) is an easy consequence of the definition of $K$ and the fact that $\{R_{h}\}$ is a disjoint family of sets. Indeed, if $\varepsilon \in R_h$ then
$K(\varepsilon) =  h \cdot  \mathds 1_{R_h}(\varepsilon) = h$.
%
 \end{proof}
Since $K$ is a decreasing function and $K(\R_+) = \mathbb N \cup \{0\}$, 
it remains true that
\[
\overline v_{K(\varepsilon)} \wc v \; \text{in $\Lrm^2(B_1;\R^{dN^k})$}, \quad \text{as $\varepsilon \to 0$}.
\]
We are now in position to calculate the $\limsup$ inequality:
\begin{multline*}
G^\varepsilon(v_{K(\varepsilon)})
= \min_{A \in \mathfrak B(B_1)} \int_{B_1} \sigma_{A} v_{K(\varepsilon)}\cdot v_{K(\varepsilon)} 
+ \varepsilon^2 \Per(A;B_1) 
 \le \int_{B_1} \sigma_{A'_{K(\varepsilon)}}v_{K(\varepsilon)}\cdot v_{K(\varepsilon)} 
+ \varepsilon^2 \Per(A'_{K(\varepsilon)};B_1)\\
 \le \int_{B_1} \sigma_{A_{K(\varepsilon)}}v_{K(\varepsilon)}\cdot v_{K(\varepsilon)} + O(K(\varepsilon)^{-1}) 
+ \varepsilon^2 L(K(\varepsilon))
 \le \int_{B_1} W(v_{K(\varepsilon)}) + O(\varepsilon) + \varepsilon.
\end{multline*}
Hence, by \eqref{4a}
\begin{align*}
 \limsup_{\varepsilon \downarrow 0} G^\varepsilon(\overline v_\varepsilon) \le  \limsup_{\varepsilon \downarrow 0} \int_{B_1} W(v_{K(\varepsilon)}) = \lim_{h \to \infty} \int_{B_1} W(v_h) = \int_{B_1} Q_{\B}W(v).
\end{align*}
This proves the upper bound inequality. 
 \end{proof}

\begin{corollary}\label{cuchara} Let $\{w_\varepsilon\} \subset \Wrm^{\A}(B_1)$ be a sequence of almost local minimizers of the sequence of functionals
\[
\{u \mapsto G^{\e}(\A u)\}.
\] 
Assume that $\{\A w_{\e}\}$ is $2$-equi-integrable in $B_s$ for every $s < 1$. Assume also that  there exists $w \in \Wrm^{\A}(B_1)$ such that
\begin{gather*}
\A w_\varepsilon \wc \A w \quad \text{in $\Lrm^2(B_1;\R^{dN^k})$}.
\end{gather*}
Then, 
\begin{gather*}
Q_{\B}W(\A w_{\e}) \wc Q_{\B}(\A w) \qquad \text{in $\Lrm^1_{\textnormal{loc}}(B_1)$}.
\end{gather*}
Moreover, $w$ is a local minimizer of $u \mapsto G(\A u)$. 
\end{corollary}
\begin{proof} 
The first step is to check that
\begin{equation}\label{step1}
Q_{\B}W(\A w_{\e}) \wc Q_{\B}(\A w) \qquad \text{in $\Lrm^1(B_s)$, for every $s < 1$}.
\end{equation}
The sequence $\A w_{\e}$ generates (up to taking a subsequence) a $\B$-free Young measure $\mu : B_1 \to Z$ so that by Theorem \ref{fundamental}, Theorem \ref{characterization} and the local $2$-equi-integrability assumption, 
\begin{equation}\label{merol1}
W(\A w_{\e}') \wc \langle \mu_x , W \rangle \ge Q_{\B}W(\A w) \quad \text{in $\Lrm^1_{\text{loc}}(B_1)$}.
\end{equation}
Fix $s \in (0,1)$ and consider the rescaled functions
\[
w^s_{\e} \coloneqq \frac{w_{\e}(sy)}{s^{k - \frac{1}{2}}}, \qquad w^s \coloneqq \frac{w(sy)}{s^{k - \frac{1}{2}}}.
\]
It is not hard to see that, because of the (almost) minimization properties of $\{w_{\e}\}$, the rescaled sequence $\{w_{\e}^s\}$ is also a sequence of almost local minimizers of the sequence of functionals $\{u \mapsto G(\A u)\}.$\footnote{This scaling has the property that $s^{N-1}J(\A w^s,A^s) =  J_{B_s(x)}(\A w,A)$.} Moreover, $\A w^s_{\e} \wc \A w^s$ in $\Lrm^2(B_1;Z)$.

From the proof of the lower bound in Theorem \ref{gamma}, we may find a 2-equi-integrable recovery sequence $\{v_\varepsilon'\}$ for $v$, i.e., such that $v_{\e}' \wc \A w^s$ and
\[
\lim_{\varepsilon \downarrow 0} G^\varepsilon(v_\varepsilon') = G(\A w^s).
\]

Recall that, by the exactness assumption of $\A$ and $\B$, there are functions $w'_{\e} \in \Wrm^{\A}(B_1)$ such that
\[
v_{\e}' = \A w_{\e} ' \qquad \text{for every $\e > 0$}.
\]
\noindent{\bf A recovery sequence with the same boundary values.} The next step is to show that one may assume, without loss of generality, that $\supp (w_\varepsilon' - w_\varepsilon^s) \subset \subset B_1$.  

We may further assume (without loss of generality) that $\{w^s_{\e}\}$ and $\{w_{\e}'\}$ are $\Wrm^{k,2}$-uniformly bounded, and that $w^s_\varepsilon - w_\varepsilon' \wc 0$ in $\Wrm^{k,2}(B_1;\R^{d})$.


Define
\[
\tilde v_{h,\varepsilon} \coloneqq \A (\varphi_h w_\varepsilon' + (1 - \varphi_h) w^s_\varepsilon) = \varphi_h \A w_\varepsilon'
+ (1 - \varphi_h)\A w^s_\varepsilon + \overbrace{\sum_{\substack{|\beta| \ge 1\\|\alpha| + |\beta| = k}} c_{\alpha\beta} \partial^\alpha (w_\varepsilon' - w^s_\varepsilon) \partial^\beta \varphi_h}^{g(h)};
\]
where, for every $h \in \mathbb N$, $\varphi_h \in \Crm^\infty(B_1;[0,1])$ with $\varphi_h \equiv 1$ in $B_{1 - 1/h}$. Since $\|g(h)\|_{\Lrm^2(B_1)} \to 0$ as $\varepsilon \to 0$, we infer that
\[
\limsup_{\varepsilon \downarrow 0} \|\tilde v_{h,\varepsilon} - \A w_{\e}'\|_{\Lrm^2(B_1)} \le 
\limsup_{\varepsilon \downarrow 0} \|\A w_\varepsilon'\|_{\Lrm^2(B_ 1 \setminus B_{1 - 1/h})}  +  \limsup_{\varepsilon \downarrow 0}\|\A w_\varepsilon\|_{\Lrm^2(B_ 1 \setminus B_{1 - 1/h})}.
\]
We now let $h \to \infty$ and use the $2$-equi-integrability of $\{\A w^s_{\e}\}$ and $\{\A w_{\e}'\}$ to get
\[
\limsup_{h \to \infty}\; \limsup_{\varepsilon \downarrow 0} \|\tilde v_{h,\varepsilon} - \A w_{\e}'\|_{\Lrm^2(B_1)} = 0.
\]
Thus, we may find a diagonal sequence $\tilde v_{\e} = \tilde v_{h(\e),\e} = \A \tilde w^s_{\e}$ which is $2$-equi-integrable, $\supp(w_{\e}^s - \tilde w_{\e}) \subset \subset B_1$, and such that
\[
\lim_{\varepsilon \downarrow 0} \|\A w_{\e}' - \A \tilde w_{\e}\|_{\Lrm^2(B_s)} = O(\e).
\]
In particular, the (almost) local minimizing property of $\{\A w^s_{\e}\}$ 
gives
\[
 \limsup_{\e \downarrow 0} \int_{B_1} W(\A w^s_{\e}) \le \limsup_{\e \downarrow 0} G^\varepsilon(\A w_{\e}^s) \le \limsup_{\e \downarrow 0} G^\varepsilon(\A\tilde w_{\e}) \le \lim_{\e \downarrow 0} G^\varepsilon(\A w_{\e}') = G(\A w^s). 
\]
Rescaling back, the inequality above yields
\[
 \limsup_{\e \downarrow 0} \int_{B_s} W(\A w_{\e}) \le \int_{B_s} Q_{\B}W(\A w),
\] 
which together with \eqref{merol1} proves \eqref{step1}.

\noindent{\bf Local minimizer of $G$.} The second step is to show that $w$ is a local minimizer of $u \mapsto G(\A u)$. We argue by contradiction: assume that $w$ is not a local minimizer of $u \mapsto G(\A u)$, then we would find $s \in (0,1)$ and $\eta \in \Crm^\infty_c(B_s;\R^{dN^k})$ for which
\[
G(\A w) > G(\A w + \A \eta).
\]
Again, using a re-scaling argument, this would imply that
\[
G(\A w^s) > G(\A w^s + \A \eta^s).
\]
Similarly to the previous step, we can find a $2$-equi-integrable recovery sequence $\{\A (\phi^s_{\e} + \eta^s)\}$ of $(\A w^s + \A \eta^s)$ with the property that $\supp(\phi^s_{\e} - w^s_{\e}) \subset \subset B_1$, for every $\e > 0$.
On the other hand, 
 the (almost) minimizing property of $\A w^s_\varepsilon$ and \eqref{step1} yield
\[
G(\A w^s +  \A \eta^s) < G(\A w^s) = \lim_{\varepsilon \downarrow 0} G^\varepsilon(\A w^s_\varepsilon) \le \lim_{\varepsilon \downarrow 0}G^\varepsilon(\A \phi^s_\varepsilon + \A \eta^s) = G(\A w^s + \A \eta^s),
\]
which is a contradiction. This shows that $w$ is a local minimizer of $u \mapsto G(\A u)$.
\end{proof}
Let us recall, for the proof of the next proposition, that the higher integrability assumption \eqref{eq:assumption} on local minimizers $\tilde u$ of $u \mapsto G(\A u)$ reads:
\[
\tag{Reg}[\A \tilde u]^2_{\Lrm^{2,N-\delta}(B_{1/2})} \le c\| \A\tilde u\|^2_{\Lrm^2}(B_1), \quad \text{for some $\delta \in [0,1)$}.
\]
\begin{proposition}\label{12} Let $(w,A)$ be a saddle-point of problem \eqref{P}. 
Assume that the higher integrability condition \eqref{eq:assumption} holds for local minimizers of $u \mapsto G(\A u)$.
Then, for every $K \subset \subset \Omega$ there exists a positive constant $C(K) > 1$ and a smallness constant
 $\rho \in (0,1/2)$ such that at least one of the following properties
\begin{itemize}
 \item[1.] $
   J_{B_r(x)}(\A w,A) \le C(K)r^{N-1},
 $
 \item[2.]$
   J_{B_{\rho r}(x)}(\A w,A) \le \rho^{N - (1 + \delta)/2} J_{B_r(x)}(\A w,A),
 $
\end{itemize}
holds for all $x \in K$ and every $r \in (0,\dist(K,\partial \Omega))$.
Here, 
\[
 J_{B_r(x)}(\A u,A) = \int_{B_r(x)} \sigma_A \A u \cdot \A u \; \dd y \; + \; \Per(A;B_r(x)),
\]
\end{proposition}
\begin{proof} Let $(w,A)$ be a saddle-point of \eqref{P} and fix $\rho \in (0,1)$ (to be specified later in the proof). We argue by contradiction through a blow-up technique: Negation of the statement would allow us to find a sequence 
$\{(x_h,r_h)\}$ of points $x_h \in K$ and positive radii $r_h \downarrow 0$  for which 
\begin{gather}
  J_{B_{r_h}(x_h)}(\A w,A) > hr_h^{N-1}, \quad \text{and}\label{joder1}\\
  J_{B_{\rho{r_h}}(x_h)}(\A w,A) > \rho^{N - (1 + \delta)/2}J_{B_{r_h}(x_h)}(\A w,A).\label{joder2}
\end{gather}

\noindent {\bf An equivalent variational problem.} It will be convenient to work with a similar variational problem: Consider the saddle-point problem  
\begin{equation*}\label{Ptilde}
\inf\left\{ \sup_{u \in \Wrm^{\A}_0(\Omega)} \tilde I_\Omega(\A u,A) : A \subset \R^N \; \text{Borel set}, \; A \cap \Omega^c \equiv A_0 \cap \Omega^c \right\}, \tag{$\tilde{\rm{P}}$}
\end{equation*}
where
\[
\tilde I_\Omega(\A u,A) \coloneqq  \int_\Omega 2\tau_A \cdot \A u \dd x - \int_\Omega \sigma_{A} \A u \cdot \A u \dd x + \Per(A;\overline \Omega).
\] 
Here we recall the notation $\tau_A \coloneqq \sigma_A\A w_A$, where $w_A \in \Wrm^{\A}_0(\Omega)$ is the unique maximizer of $u \mapsto I_\Omega(u,A)$. It follows immediately from the identity 
\[
\int_\Omega \tau_{A} \cdot \A u \dd x = \int_\Omega f u \dd x \qquad u \in \Wrm^{\A}_0(\Omega),
\]
that saddle-points $(w,A)$ of problem \eqref{P} are also saddle-points of \eqref{Ptilde} and vice versa; hence, in the following we will make no distinction between saddle-points of \eqref{P} and \eqref{Ptilde}. A special property of $\tilde I$ is that, locally, it is always positive on saddle-points $(w,A)$ of \eqref{P}. Indeed, in this case $w = w_A$ and therefore
\begin{equation}\label{transition}
\tilde I_{B_r(x)}(\A w,A) = \int_{B_r(x)} \sigma_A \A w_A \cdot \A w_A + \Per(A;B_r(x)) = J_{B_r(x)}(\A w,A), \qquad B_r(x) \subset \Omega.
\end{equation}

\noindent {\bf A re-scaling argument.} 
We re-scale and translate $B_r(x)$ into $B_1$ by letting
\begin{equation}\label{blow2}
 A^{r,x} \coloneqq \frac{A}{r} - x, \quad f^{r,x}(y)\coloneqq r^{k + \frac{1}{2}}f(ry - x) \to 0 \text{ in $\Lrm^\infty(B_1)$},
\quad \text{and} \quad w^{r,x}(y) \coloneqq
 \frac{w(ry - x)}{r^{k - \frac{1}{2}}}.
 \end{equation}
 A further normalization on the sequence takes place by setting 
 \[
 \varepsilon(h)^2 \coloneqq r^{N-1} \cdot J_{B_{r_h}(x_h)}(\A w,A)^{-1} = O(h^{-1}),
 \]
 and defining
 \[
 A_{\varepsilon(h)} \coloneqq A^{r_h,x_h}, \quad f_{\varepsilon(h)} \coloneqq \varepsilon(h) \cdot f^{r_h,x_h},
  \quad w_{\varepsilon(h)} \coloneqq \varepsilon(h) \cdot w^{r_h,x_h}, \quad \text{and} \quad \tau_{\varepsilon(h)} \coloneqq 
  \sigma_{A_{\varepsilon(h)}} \A w_{\varepsilon(h)}.
 \]


It is easy to check that the scaling rule \eqref{scaling}, and the relations \eqref{joder2} and \eqref{transition} imply
\begin{gather}\label{papa2}
 J^{\varepsilon(h)}(\A w_{\varepsilon(h)},A_{\varepsilon(h)}) = 1, \quad \text{and}\\
  \int_{B_\rho} \sigma_{A_{\varepsilon(h)}} \A w_{\varepsilon(h)} \cdot \A w_{\varepsilon(h)} + 
  {\varepsilon(h)}^2 \Per(A_{\varepsilon(h)};B_\rho) > \rho^{N - (1 + \delta)/2}. 
  \label{papa3}
\end{gather}
In particular, due to the coercivity of $\sigma_1$ and $\sigma_2$, the norms $\| \A w_{\varepsilon(h)}\|^2_{\Lrm^2(B_1)}$ are $h$-uniformly bounded by $M$.

\noindent{\bf Local almost-minimizers of $G^{\varepsilon(h)}$.} The next step is to show that $\{w_{\varepsilon(h)}\}$ is $O(\e)$-close in $\Lrm^2$ to a sequence $\{\tilde w_{\e}\}$ of almost minimizers of $\{u \mapsto G^{\varepsilon(h)}(\A u)\}$. 
Observe that $w_{\varepsilon(h)}$ is the unique solution to the equation
\[
\A^*(\sigma_{A_\varepsilon} \A u) = f_{\varepsilon(h)}, \qquad \text{$u \in \Wrm^{\A}_{w_{\varepsilon(h)}}(B_1)$}. 
\]
Let $\tilde w_{\varepsilon(h)}$ be the unique minimizer of $u \mapsto J^{\varepsilon(h)}(\A u,A_{\varepsilon(h)})$ -- see \eqref{I} -- in the affine space 
$W^{\A}_{w_{\varepsilon(h)}}(B_1)$. Thus, in particular, $\tilde w_{\varepsilon(h)}$ is the unique solution of the equation
\[
\A^* (\sigma_{A_{\varepsilon(h)}} \A u) = 0, \qquad \text{$u \in\Wrm^{\A}_{w_{\varepsilon(h)}}(B_1)$}.
\]
A simple integration by parts, considering that $\tilde w_{\e}(h) -w_{\e}(h) \in \Wrm^{\A}_0(B_1)$,  gives the estimate
\begin{equation}\label{eq:orden}
\|\A w_{\varepsilon(h)} - \A \tilde w_{\varepsilon(h)}\|^2_{\Lrm^2(B_1)} \le C(B_1) \cdot M^2\|f_{\varepsilon(h)}\|^2_{\Lrm^2(B_1)} = O(h^{-1}),
\end{equation}
where $C(B_1)$ is the Poincar\'e constant from \eqref{poincare2}; and therefore $\|w_{\varepsilon(h)} - \tilde w_{\varepsilon(h)}\|_{\Wrm^{k,2}_0(B_1)} = O(h^{-1})$. 

Lastly, we use strongly the fact that $(w,A)$ is a saddle-point of \eqref{P} to see that $\{(w_{\varepsilon(h)},A_{\varepsilon(h)})\}$ is also a {\it local} saddle-point of the energy
\[
(u,E) \mapsto \tilde I^{ \varepsilon(h)}(\A u,E) \coloneqq \int_{B_1} 2 \tau_E \cdot \A u \dd y  -  \int_{B_1} \sigma_E \A u \cdot \A u \dd y \; + \; \varepsilon(h)^2\Per(E;B_1).
\]
Moreover, by \eqref{scaling}, \eqref{transition} and \eqref{eq:orden} one has that 
\begin{equation}\label{eq:juntos}
\tilde I^{ \varepsilon(h)}(\A w_{\varepsilon(h)},A_{\varepsilon(h)}) = J^{\varepsilon(h)}(\A w_{\varepsilon(h)},A_{\varepsilon(h)}) = J^{\varepsilon(h)}(\A \tilde w_{\varepsilon(h)},A_{\varepsilon(h)}) + O(h^{-1}).
\end{equation}
An immediate consequence of the two facts above is that $\{\tilde w_{\varepsilon(h)}\}$ is a sequence of {\it local}  almost minimizers of the sequence of functionals $\{u \mapsto G^{\varepsilon(h)}(\A u)\}$. The local (almost) minimizing properties of the sequence 
$\{\tilde w_{\varepsilon(h)}\}$ -- with respect to the functionals $\{u \mapsto G^{\varepsilon(h)}(\A u)\}$ -- are not affected by subtracting~{$\A$-free} fields; hence, using the compactness assumption of $\A$ once more, we may assume without loss of generality that $\sup_h \|\tilde w_{\varepsilon(h)}\|_{\Wrm^{k,2}(B_1)} < \infty$. Upon passing to a further subsequence, we may also assume that there exists $\tilde w \in 
\Wrm^{k,2}(B_1)$ such that
\[
\tilde w_{\varepsilon(h)} \wc \tilde w \quad \text{in $\Wrm^{k,2}(B_1;\R^d)$}.
\]

\noindent{\bf Equi-integrability of $\{\A \tilde w_{\e(h)}\}$.} The last but one step is to show that $\{\A \tilde w_{\e}\}$ is a $2$-equi-integrable sequence in $B_s$, for every $s < 1$. 

Since $\sigma_{A_\varepsilon}$ is uniformly bounded, there exists $\tilde \tau \in \Lrm^2(B_1;\R^{dN^k})$ such that (upon passing to a further subsequence)
\begin{equation}\label{integration}
\sigma_{A_{\varepsilon(h)}}\mathcal A \tilde w_{\varepsilon(h)} =: \tilde \tau_{\varepsilon(h)} \wc \tilde \tau \quad \text{in }  \Lrm^2(B_1;\R^{dN^k}), \quad \mathcal A^* \tilde \tau_{\varepsilon(h)} = \A^* \tilde \tau = 0. 
\end{equation}
Let $\varphi \in \mathcal D(B_1)$ and fix $\varepsilon > 0$, integration by parts yields 
\[\langle  \tilde \tau_{\varepsilon(h)}\cdot \mathcal A \tilde w_{\varepsilon(h)}, \varphi\rangle = 
 - \sum_{\substack{|\beta| \ge 1\\|\alpha| + |\beta| = k}} c_{\alpha\beta} \langle \tilde \tau_{\varepsilon(h)},
 \partial^\alpha \tilde w_{\varepsilon(h)} \partial^\beta \varphi
\rangle \qquad  \quad c_{\alpha,\beta} \in \R.
\]
Since the term in the right hand side of the equality depends only on $\nabla^{k-1} \tilde w_{\varepsilon(h)}$, the strong convergence $\tilde w_{\e} \to \tilde w$ in $\Wrm^{k-1,2}(B_1;\R^d)$ gives
\[
\lim_{\varepsilon \to 0} \langle  \tilde \tau_{\varepsilon(h)}\cdot \mathcal A \tilde w_{\varepsilon(h)}, \varphi\rangle =
 - \sum_{\substack{|\beta| \ge 1\\|\alpha| + |\beta| = k}} c_{\alpha\beta} \langle \tilde \tau,
 \partial^\alpha \tilde w \partial^\beta \varphi
\rangle
=\langle  \tilde \tau\cdot \mathcal A \tilde w, \varphi\rangle.
\]
Therefore, 
\[
\sigma_{A_{\varepsilon(h)}} \A \tilde w_{\varepsilon(h)} \cdot \A \tilde w_{\varepsilon(h)} = \tilde \tau_{\varepsilon(h)} \cdot \A \tilde w_{\varepsilon(h)} \toweakstar \tilde \tau \cdot \A \tilde w \in \Lrm^1(B_1) \qquad \text{weakly* in $\M^+(B_1)$}.
\]
The positivity of $\sigma_{A_\varepsilon}\mathcal A \tilde w_\varepsilon\cdot\mathcal A \tilde w_\varepsilon$, the Dunford-Pettis Theorem and the convergence above imply that the sequence
\[\{\sigma_{A_\varepsilon}\mathcal A \tilde w_\varepsilon\cdot\mathcal A \tilde w_\varepsilon\} \qquad \text{is equi-integrable in $B_s$; for every $s < 1$}.\]
In turn, due to the uniform coerciveness and boundedness of $\{\sigma_{A_\varepsilon}\}$, both sequences $\{\mathcal A \tilde w_\varepsilon\}$ and
$\{\tilde \tau_\varepsilon\}$ are $2$-equi-integrable in $B_s$; for every $s < 1$.\\

\noindent {\bf The contradiction.} We are in position to apply Proposition \ref{cuchara} to the sequence $\{\tilde w_{\e}\}$, which in particular implies
\begin{equation}\label{merol2}
\begin{split}
\e(h)^2\Per(A_{\e(h)};B_\rho) & \to 0,\\
\sigma_{A_{\varepsilon(h)}} \A \tilde w_{\varepsilon(h)} \cdot \A \tilde  w_{\varepsilon(h)}   & \wc Q_{\B} W(\A \tilde w) \le M|\A \tilde w|^2, \quad \text{in $\Lrm^1_{\text{loc}}(B_1)$},
\end{split}
\end{equation}
and that $w$ is a local minimizer of $u \mapsto G(\A u)$.
On the other hand, the higher integrability assumption \eqref{eq:assumption} tells us that 
\begin{equation}\label{eq:contramadre}
\begin{split}
 [\A \tilde w]^2_{\Lrm^{2,\,N- \delta} (B_{1/2})}&  \le c\|\A \tilde w \|_{\Lrm^2(B_1)}^2.
\end{split}
\end{equation}
We set the value of $\rho \in (0,1/2)$ to be such that $2 c M^2 \rho^{(1 - \delta)/2} \le 1$. Taking the limit in \eqref{papa2} and \eqref{papa3}, using Fatou's Lemma, \eqref{eq:orden}, \eqref{eq:juntos}, \eqref{merol2}  and \eqref{eq:contramadre}, we get
\begin{align*}
\frac{1}{M} \|\A \tilde w\|^2_{\Lrm^2(B_1)} &  \le \lim_{h \to \infty} J^{\varepsilon(h)}(\A \tilde w_{\varepsilon(h)},A_{\varepsilon(h)}) = 1 
\\
& \le \bigg(\frac{1}{\rho^{N - (1 +\delta)/2}}\bigg)  \|Q_{\B}W(\A \tilde w)\|_{\Lrm^1(B_\rho)} \le \bigg(\frac{M \rho^{(1 -\delta)/2}}{\rho^{N - \delta}}\bigg) \|\A \tilde w\|^2_{\Lrm^2(B_\rho)} \\
& \le M \rho^{(1 -\delta)/2} [\A \tilde w]^2_{\Lrm^{2,\,N- \delta}(B_{1/2})}\le c M \rho^{(1 -\delta)/2}\|\A \tilde w \|_{\Lrm^2(B_1)}^2 \\
& \le \frac{1}{2M}\|\A \tilde w \|_{\Lrm^2(B_1)}^2;
\end{align*}
a contradiction.
 \end{proof}

\begingroup
\def\thetheorem{\ref{thm2}}
\begin{theorem}[upper bound]\label{thm:CE}
Let $(w,A)$ be a variational solution of problem \eqref{P}.
Assume that
the higher integrability condition 
 \[
  [\A \tilde u]^2_{\Lrm^{2,N-\delta}(B_{1/2})} \le c\|\A \tilde u \|_{\Lrm^2(B_1)}^2, \quad \text{for some $\delta \in [0,1)$ and some positive constant $c$},
 \]
 holds for local minimizers of the energy $u \mapsto \int_{B_1} Q_{\B} W(\A u)$, where $u \in \Wrm^{\A}(B_1)$.
  Then, for every compactly contained set $K \subset \subset \Omega$, there exists a positive constant $ \Lambda_K$ such that
\begin{equation}\label{critical}
 \int_{B_r(x)} \sigma_A \A w \cdot \A w \; \dd y \; + \; \Per(A;B_r(x)) \le \Lambda_K r^{N-1} \qquad \forall \;x \in K, \forall \;r \in (0,\dist(K,\partial \Omega)).
\end{equation}
\end{theorem}
\addtocounter{theorem}{-1}
\endgroup
%
\begin{proof}
Let $x \in K$, and set  
\[ \varphi(r,x) \coloneqq J_{B_r(x)}(\A w,A),
\]
where we recall that
\[
J_{B_r(x)}(\A w,A) = \int_{B_r(x)} \sigma_A \A w \cdot \A w  \; \dd y \; + \; \Per(A;B_r(x))
\]
Proposition \ref{12} tells us that there exists a positive constant $\rho \in (0,1/2)$ such that if $B_r(x) \subset \Omega$, then
\[
\varphi(\rho r,x) \le \rho^{N - (1 + \delta)/2}\varphi(r,x) + C(K) r^{N-1}.
\]
An application of the Iteration Lemma~\cite[Lem. 2.1, Ch. III]{GiaquintaBook83} (stated below) to
$r \in (0, \min\{1,\text{dist}(K,\partial \Omega\})$, and
$\alpha_1 \coloneqq N - (1 + \delta)/2 > \alpha_2 \coloneqq N-1$ yields the existence of positive constants $c = c(x)$, and $r = r(K)$ such that 
\[
\varphi(s,x) \le cs^{N-1} \qquad \forall \; s \in (0,R(K)).
\]
Notice that the constants $c$ and $r$ depend continuously on $x \in \Omega$. Hence, for any 
$K \subset \subset \Omega$ we may find $\Lambda_K > 0$ for which 
\begin{equation*}\label{local12}
J_{B_r(x)}(\A w,A) \le \Lambda_Kr^{N-1} \qquad \forall \; x \in K, \; \forall 	\; r \in (0,\dist(K,\partial \Omega)).
\end{equation*}
 \end{proof}
\begin{lemma}[Iteration Lemma]\label{iteration}Assume that $\varphi(\rho)$ is a non-negative, real-valued, non-decreasing function defined on the $(0, 1)$ interval. Assume further that there exists a number $\tau \in  (0, 1)$ such that for all $r < 1$ we have
$$\varphi(\tau r)\le\tau^{\alpha_1}\varphi(r)+ Cr^{\alpha_2}$$
for some non-negative constant $C$, and positive exponents $\alpha_1 > \alpha_2$. Then there exists a positive constant $c=c(\tau,\alpha_1,\alpha_2)$ such that for all $0\le\rho\le r\le R$ we have
$$\varphi(\rho) \le c\left(\frac{\rho}{r}\right)^{\alpha_2} \varphi(r) + C 	\rho^{\alpha_2}.$$
\end{lemma} 
\begin{corollary}[compactness of blow-up sequences]\label{blow} Let $(w,A)$ be a variational solution of problem \eqref{P}. 
Under the assumptions of the upper bound Theorem \ref{thm:CE}, 
there exists a positive constant $C_K$ such that
\begin{equation}\label{eq:uniform}
  [\A w]^2_{\Lrm^{2,N-1}(K)} \le C_K.
\end{equation}
\end{corollary}
\begin{proof}
The assertion follows directly from the Upper Bound Theorem and the coercivity of $\sigma_1$ and $\sigma_2$. 
  \end{proof}

\section{The Lower Bound: proof of estimate (\ref{lower})}\label{LBo}

During this section we will write $(w,A)$ to denote a solution of problem \eqref{P} under the assumptions of
Theorem~\ref{thm:CE}. In light of the results obtained in the previous section we will assume, throughout the rest of the paper, that for every compact set $K \subset \subset \Omega$ there exist positive constants $C_K$, and $\Lambda_K$ such that
\begin{gather*}
\Per(A;B_r(x)) \le \Lambda_K r^{N-1},\\
\| \A w^{x,r}\|^2_{\Lrm^2}(B_1) \le [\A w]^2_{\Lrm^{2,N-1}(K)} \le C_K,
\end{gather*}
for all $x \in K$ and every $r \in (0,\dist(K,\partial \Omega))$. 

The main result of this section is a lower bound on the density of the perimeter in $\partial^* A$. In other words, there exists a positive constant $\lambda_K = \lambda_K(N,M)$ such that
\begin{equation*}\label{LB}
\Per(A;B_r(x)) \ge \lambda_K r^{N-1} \quad \text{for every $0 < r <   \dist(x,\partial \Omega)$}. \tag{LB}
\end{equation*}

There are two major consequences from estimate \eqref{LB}. The first one (cf. Corollary \ref{essential}) is that the difference 
between the topological boundary of $A$ and the reduced boundary of $A$ is at most a set of zero $\mathcal H^{N-1}$-measure. In other words, $(\partial A \setminus \partial^* A) = \Sigma$ where $\mathcal H^{N-1}(\Sigma) = 0$ (cf. 
\cite[Theorem 2.2]{AmbrosioButtazzo93}). 
The second implication is that \eqref{LB} is a necessary assumption for the Height bound Lemma and the Lipschitz approximation Lemma, which are essential tools to prove the flatness excess improvement in the next section.
 
Throughout this section and the rest of the manuscript we will constantly use the following notations:

The scaled Dirichlet energy
\[
D(w;x,r) := \frac{1}{r^{N-1}}\int_{B_r(x)} |\mathcal A w|^2 \, \text{d}y,
\]
and the excess for $\gamma$-weighted energy 
\[
E_\gamma(w,A;x,r) := D(w;x,r) + \frac{\gamma}{r^{N-1}}\Per(A,B_r(x)).
\]
Granted that the spatial-, radius-, or $(w,A)$- dependence is clear, we will shorten the notations to the only relevant variables, e.g., $D(r)$ and $E_{\gamma}(r)$. 
Recall that, up to translation and re-scaling, we may assume 
\[
0 \in \partial^* A \cap K, \quad  \text{and} \quad  B_1 \subset K + B_9 \subset \Omega.
\] 
Bear also in mind that all the constants in this section are universal up to their dependence on $\Lambda_K$ and $C_K$. 

We will proceed as follows. First we prove in Lemma \ref{lowperi} that if the density of the perimeter is sufficiently small,  one may regard the regularity properties
of solutions as those ones for an elliptic equation with constant coefficients. Then, in Lemma \ref{mueller}, we prove a lower bound on the decay of the density of the perimeter in terms of $D$. Combining these results,
we are able to show a discrete monotonicity formula on the decay of~$E_{\gamma}$. 

The proof of the lower density bound (\ref{LB}) follows easily from this discrete monotonicity formula,
De Giorgi's Structure Theorem, and the upper bound Theorem of the previous section. 
Finally, we prove that the difference between $\partial A$ and $\partial^* A$ is $\mathcal H^{N-1}$-negligible (Theorem \ref{essential}) as a corollary of the estimate \eqref{LB}. 

\begin{lemma}[approximative solutions of the constant coefficient problem]\label{lowperi}
 For every $\theta_1 \in (0,1/2)$, there exist positive constants\footnote{As it can be seen from the proof of Lemma \ref{lowperi}, the constant $c_1$ does not depend on $K$.} $c_1(\theta_1,N,M)$ and $\varepsilon_1(\theta_1,N,M)$ such that either
 \[
 \int_{B_\rho} |\mathcal A w|^2 \dd y \le c_1 \rho^N \|f\|_{\Lrm^\infty(B_1)}^2,
 \]
  or  \[
  \int_{B_\rho} |\mathcal A w|^2 \dd y \le 2c \rho^{N}\int_{B_1} |\mathcal A w|^2 \dd y 
   \quad \text{for every $\rho \in
  [\theta_1,1)$},
 \]
  where $c = c(N,M)$ is the constant from Lemma \ref{lem:cosntant};
whenever 
\[
\Per(A;B_1) \le \varepsilon_1.\]
\end{lemma}
\begin{proof} Since $c \ge 2^N$, the result holds if we assume $\rho \ge 1/2$, therefore
we focus only on the case where $\rho \in (\theta_1,1/2)$. Fix $\theta_1 \in (0,1/2]$. We argue by contradiction:
We would find a sequence of pairs $(w_h,A_h)$ (locally solving \eqref{P} in $B_1$ for a source function $f_h$) and constants $\rho_h \in [1/2,\theta_1]$, such that \begin{equation}\label{contrad}
  \delta_h^2 \coloneqq \int_{B_{\rho_h}} |\mathcal A w_h|^2 \, \text{d}y > 2 \, c  \rho_h^N
  \int_{B_1} |\mathcal A w_h|^2 \, \text{d}y,
 \end{equation}
and simultaneously 
\[
\rho^N_h\cdot\frac{\|f_h\|^2_{\Lrm^\infty(B_1)}}{\delta^2_h} \le \frac{1}{h}, \quad \text{and} \quad  \Per(A_h;B_1) \le \frac{1}{h}.
\]

The estimate above yields $\delta_h^{-1} f_h \to 0$ in $\Lrm^2(B_1;\R^d)$. Also, since $\Per(A_h;B_1) \to 0$, the isoperimetric inequality yields that
either $\sigma_{A_h} \to \sigma_1$ or $\sigma_{A_h} \to \sigma_2$
in $\Lrm^2$ as $h$ tends to infinity.
Let us assume that the former convergence $\sigma_{A_h} \to \sigma_1$ holds. 

Let $u_h \coloneqq \delta_h^{-1} w_h$, for which 
\[
\sup_h \|\A u_h\|_{\Lrm^2(B_1)} < \infty.
\] 
We use that $w_h$ is a (local) solution to \eqref{P} for $A_h$ as indicator set and $f_h$ as source term, to see that
\[
\A^*(\sigma_{A_h} \A u_h) = \delta_h^{-1}f_h \qquad \text{in $B_1$}.
\]
Up to passing to a further subsequence, we may assume that $u_h \wc u$ in $\Wrm^{k,2}(B_1;\R^{dN^k})$.
We may then apply the compensated compactness result from Lemma \ref{compensated} to obtain that
\[
 \A^*(\sigma_1 \A  u) = 0
\quad \text{in $B_1$},\]
and 
\[
 D(u_h;s) \to D(u;s) \qquad \text{where $\rho_h \to s \in [\theta_1,1/2]$}.
\]
Hence, by (\ref{contrad}) and Fatou's Lemma one gets
\[
 2 \, c s^N D(u;1) \le 
 \lim_{h \to \infty} c\rho^N_h D(u_h;1) \le 1 =
 \lim_{h \to \infty} D(u_h;\rho_h)
= \lim_{h \to \infty} D(u_h;s) = D(u;s).
\]
This is a contradiction to Lemma \ref{lem:cosntant} because $u$ is a solution for the problem with constant coefficients $\sigma_1$. The case when $\sigma_{A_h} \to \sigma_2$ can be solved by similar arguments.
 \end{proof}
The next lemma is the principal ingredient in proving the \eqref{LB} estimate. It relies on a cone-like comparison to show that 
the decay of the perimeter density is controlled by $D(r)/r$: The perimeter density cannot blow-up at smaller scales, while for a fixed scale, the perimeter density is small.
\begin{lemma}[universal comparison decay]\label{mueller} There exists a positive constant\footnote{The constant $c_2$ is independent of the compact set $K$; indeed, this is the result of universal comparison estimates in $\Omega$.} $c_2 = c_2(N,M)$ such that 
$$\frac{\dd}{\dd r}\bigg|_{\rho = r}\left(\frac{\Per(A;B_\rho)}{\rho^{N-1}}\right) \ge -c_2\frac{D(r)}{r} \qquad \text{for a.e.} \; r \in (0,1].$$
\end{lemma}
\begin{proof} For a.e. $r \in (0,1)$ the slice $\langle A,g  ,r \rangle$, where $g(x) = | x |$,
is well defined (see Section \ref{sec:gmt}). Fix one such $r$ and let $\tilde{A}$ be the cone-like comparison set to $A$ as in \eqref{eq:cone}. By minimality of $(w,A)$ and a duality argument, we get 
$$\int_{B_r} \sigma_A^{-1} \tau_A \cdot \tau_A \dd y + \Per(A;B_r)  \le \int_{B_r} \sigma^{-1}_{\tilde A} \tau_A \cdot \tau_A \dd y + \Per(\tilde{A};B_r)$$
for $\tau_A = \sigma_A\A w$. Hence,
\begin{equation}\label{6.1}
\begin{split}
\Per(A;B_r) & \le \Per(\tilde{A};B_r) + M \int_{B_r}  |\A w_A|^2 \, \text{d}y \\
& \le \frac{r}{N-1}\langle A,g,r \rangle(\mathbb R^N) + M^3 r^{N-1}D(r).  
\end{split}
\end{equation}

To reach the inequality in the last row we have used that the cone extension $\tilde A$ is precisely built  (cf.  \eqref{eq:conelike}) so that the Green-Gauss measures $\mu_{\tilde A}$ and $ \mu_{A}$ agree in $(B_r)^c$; where, by \eqref{eq:conegrowth},
$$\Per(\tilde A;B_{\rho}) = \frac{1}{(N-1)}\left(\frac{\rho^{N-1}}{r^{N-2}}\right)\mathcal H^{N-2}(\partial^*A \cap \{g = r\})\le  \frac{1}{(N-1)}\left(\frac{\rho^{N-1}}{r^{N-2}}\right)\langle A,g,r \rangle(\mathbb R^N)$$ 
for all $0 < \rho \le r$.
We know from \eqref{eq:slice}  that $\frac{\dd}{\dd\rho}\big|_r\Per(A;B_\rho) \ge \langle A,g,r \rangle(\mathbb R^N)$ for a.e. $r > 0$. Since \eqref{6.1} and the previous inequality are valid almost everywhere in $(0,1)$, a combination of these arguments yields
$$ \frac{\dd}{\dd r}\bigg|_{\rho = r}\left(\frac{\Per(A;B_\rho)}{\rho^{N-1}}\right) \ge -M^3(N-1)\frac{D(r)}{r} 
\qquad \text{for a.e. $r \in (0,1)$}.$$
The result follows for $c_2 := M^3(N-1)$.
 \end{proof}
The following result is a discrete monotonicity for the weighted 
excess energy $E_\gamma$. We remark that, in general, a monotonicity formula may not be expected in the case of systems. 
\begin{theorem}[Discrete monotonicity]\label{eximp2}
There exist positive constants $\gamma = \gamma(N,M)$, $\varepsilon_2 = \epsilon_2(\gamma,N) \le \mathrm{vol}(B_1') \cdot \gamma /2$, 
 and $\theta_2 = \theta_2(N,M) \in (0,1/2)$ such that
 \begin{equation}\label{eximp}
  E_\gamma(\theta_2) \le E_\gamma(1) + c_1(\theta_2)\|f\|_{\Lrm^\infty(B_1)}^2, \qquad \text{whenever} \quad E_\gamma(1) \le \varepsilon_2.
 \end{equation} 
 \end{theorem}
 \begin{proof}
  We fix $\gamma$  and $\theta_1$ such that 
  \[
  \gamma c_2 \max\{c,c_1(\theta_1)\}\le \frac{1}{4}, \quad \text{where} \quad 2 \theta_1 c \le \frac{1}{2}.
  \]
Set $\theta_2 := \theta_1$. Recall that $c_2$ is the constant from Lemma \ref{mueller}, and $c$ is the constant of Lemma \ref{lem:cosntant}. 

Let also
 $\varepsilon_2 = 
\varepsilon_2(\gamma,\varepsilon_1)$ be a positive constant with $\varepsilon_2 \le \min\{\gamma \varepsilon_1(\theta_2),
\gamma \cdot \text{vol}(B_1') /2\}$. This implies 
 \[
  \Per(A;B_{1}) \le \varepsilon_1(\theta_2),
 \]
which in turn gives for $c_1 = c_1(\theta_2)$ (see Lemma \ref{lowperi})
\[
 E_\gamma(\theta_2) \le \frac{\gamma}{\theta_2^{N-1}}\Per(A;B_{\theta_2}) + 2 c \theta_2D(1) + c_1\theta_2\|f\|^2_{\Lrm^\infty(B_1)}.
\]
Now, we apply Lemma \ref{lowperi} and Lemma \ref{mueller} to $s \in (\theta_2,1)$ to get
\begin{align*}
 E_\gamma(\theta_2) & \le  \frac{\gamma}{\theta_2^{N-1}}\Per(A;B_{\theta_2}) + 2 c \theta_2D(1) + c_1\theta_2\|f\|^2_{\Lrm^\infty(B_1)}\\
 & \le \gamma \,\Per(A;B_1) + \gamma \int_{\theta_2}^1 -\frac{\dd}{\dd r}\Big|_{r=s}\left(\frac{\Per(A,B_r)}{r^{N-1}}\right) \dd s + \frac{1}{2}D(1) + c_1\theta_2\|f\|^2_{\Lrm^\infty(B_1)}\\
 &\le \gamma \, \Per(A;B_1) + \gamma c_2\int_{\theta_1}^1 \frac{D(s)}{s} \dd s + \frac{1}{2}D(1)  + c_1\theta_2\|f\|^2_{\Lrm^\infty(B_1)}\\
 & \le \gamma \, \Per(A;B_1) + 2 \gamma  c c_2 D(1) + \gamma c_2 c_1\|f\|^2_{\Lrm^\infty(B_1)} + \frac{1}{2}D(1) + c_1\theta_2\|f\|^2_{\Lrm^\infty(B_1)}\\
 & \le \gamma \, \Per(A;B_1) + D(1) + c_1\|f\|^2_{\Lrm^\infty(B_1)}\\
 & = E_\gamma(1) + c_1\|f\|^2_{\Lrm^\infty(B_1)}.
\end{align*}
This proves the desired result.
 \end{proof}

\begin{lemma}\label{merol4} For every $\e > 0$, there exist positive constants $\theta_0(N,M,K,\e) \in (0,1/2)$ and $\kappa(N,M,K,\e) > 0$ such that
\[
E_\gamma(\theta_0) \le \e + c_1\|f\|_{\Lrm^\infty(B_1)}^2;
\]
whenever
\[
\Per(A;B_1) \le \kappa.
\]
\end{lemma}
\begin{proof}
The result follows by taking $\theta_0$ such that $2c\theta_0 C_K \le \e/2$ (recall that, $D(s) \le C_K$ for every $s \in (0,1)$) and $\kappa \le \min\bigg\{\frac{\e\theta^{N-1}_0}{2\gamma},\e_1(\theta_0)\bigg\}$ and then simply applying Lemma \ref{lowperi}.
\end{proof}

\begin{lemma}\label{merol3} Let $(w,A)$ be a saddle-point of \eqref{P} and let $x \in K \subset \subset \Omega$. Then, for every $\e > 0$ there exists a positive radius $r_0 = r_0(N,M,K,\|f\|_{\Lrm^\infty(B_1)},\e)$ for which 
\[
E_\gamma(w,A;x, r) \le 2\e;
\]
whenever $r \le r_0$ and $\Per\big(A;B_{\theta_0^{-1}r}\big) \le \kappa(\e) \cdot \big(\frac{r}{\theta}\big)^{N-1}$.
\end{lemma}
\begin{proof} Let $r_0$ be a positive constant such that $c_1 r_0^{2k+1}\|f\|_{\Lrm^\infty(B_1)}^2 \le \theta_0^{2k +1}\e$ and let us set $s \coloneqq \theta_0^{-1}r$. Since 
\[\Per (A^{x,s};B_1) = s^{-(N-1)}\Per(A;B_{s}) \le \kappa(\e),\]
 it follows from the previous lemma and a rescaling argument that
\[
E_\gamma(w,A;r) = E_\gamma(w,A;\theta_0 s) \le \e + c_1\|f^{s}\|_{\Lrm^\infty(B_1)}^2 =
\e + c_1s^{2k +1}\|f\|_{\Lrm^\infty(B_1)}^2 \le 2\e. 
\]
\end{proof}

\begin{theorem}[lower bound]\label{jeb} Let $(w,A)$ be a solution of problem \eqref{P} in $\Omega$.
Let $K \subset \subset \Omega$ be a 
compact subset. Then, there exist positive constants $\lambda_K$ and $r_K$ depending only on $K$, 
the dimension $N$, the constant $M$ in the assumption \eqref{eq:ellipticity}, and $f$ such that 
\begin{equation*}
\Per(A;B_r(x)) \ge \lambda_Kr^{N-1},\tag{LB}
\end{equation*}
for every $r \in (0,r_K)$ and every $x \in \partial^*A \cap K$.
\end{theorem}
\begin{proof}
Let $p(\theta_2) \coloneqq \sum_{h = 0}^\infty \theta_2^{(2k + 1)h} \in \R$ and define $r_1 \in (0,1)$ to be a positive constant for which
\[
r_1^{2k +1}c_1(\theta_2) p(\theta) \|f\|_{\Lrm^\infty(B_1)}^2 \le \frac{\e_2}{4}.
\]
We argue by contradiction. If the assertion does not hold, we would be able to find a point $x \in \partial^*A$ and  a radius $r \le \min\{r_0,r_1\}$ for which 
\[\Per\big(A;B_{\frac{r}{\theta_0}}(x)\big) \le \bigg(\frac{r}{\theta_0}\bigg)^{N-1} \kappa(\varepsilon), \qquad \e \coloneqq \frac{\e_2}{4}.
\] 
After translation, we may assume that $x =0$. The fact that $r \le r_0$ and Lemma \ref{merol3} yield the estimate 
\[
E_\gamma(w,A;r) \le 2\e \le \frac{\e_2}{2};
\]
in return, Lemma \ref{eximp2} and a rescaling argument give (recall that $f^r(y) = r^{k + \frac{1}{2}}f(ry)$)
\[
E_\gamma(w,A;\theta_2 r) \le E_\gamma(w^r,A^r;1) + c_1\|f^r\|_{\Lrm^\infty(B_1)}^2 \le \frac{\e_2}{2}  + c_1r^{2k + 1}\|f\|_{\Lrm^\infty(B_1)}^2 \le \e_2.
\]
A recursion of the same argument gives the estimate
\[
E_\gamma(w,A;\theta_2^j r) \le E_\gamma(w,A;r) + c_1r^{2k + 1}\|f\|_{\Lrm^\infty(B_1)}^2 \bigg(\sum_{h = 0}^j \theta_2^{(2k + 1)h}\bigg) \le \e_2.
\]
Taking the limit as $j \to \infty$ we get
\[
\limsup_{j \to \infty} \; \frac{\Per(A;B_{\theta_2^jr})}{\text{vol}(B_1')\cdot(\theta_2^jr)^{N-1}} \le 
\limsup_{j \to \infty} \frac{E_\gamma(w,A;\theta_2^j r)}{\text{vol}(B_1') \cdot \gamma} \le \frac{\e_2}{\text{vol}(B_1') \cdot \gamma} \le \frac{1}{2}.
\]
This a contradiction to the fact that $x = 0 \in \partial^*A$ (cf. Section \ref{sec:gmt})

\end{proof}

\begin{corollary}\label{extender}
Let $(w,A)$ be a solution for problem \eqref{P} in $\Omega$. 
Let $K \subset \subset \Omega$ be a  
compact subset. Then, there exist positive constants $\lambda_K$ and $r_K$ depending only on $K$, 
the dimension $N$, and $f$ such that 
\begin{equation*}
 \Per(A;B_r(x)) \ge \lambda_Kr^{N-1},
\end{equation*}
for every $r \in (0,r_K)$ and for every $x \in \partial A \cap K$.
\end{corollary}
\begin{proof}
The property \eqref{LB} from the Lower Bound Theorem is a topologically closed property, i.e., it extends to $\overline{\partial^*A} = \supp\mu_A = \partial A$ (cf. \eqref{eq:closed}).
 \end{proof}

\begin{corollary}\label{essential} Under the same assumptions of Theorem \ref{jeb}, the following characterization for the to\-po\-lo\-gi\-cal boundary of $A$ holds:
\[
\partial A = \partial^* A \cup \Sigma, \quad \text{where $\mathcal H^{N-1}(\Sigma) = 0$}.
\]
\end{corollary}
\begin{proof}
An immediate consequence of the previous corollary is that $\mathcal H^{N-1}\llcorner \partial A \ll |\mu_A|$ as measures in $\Omega$. The assertion follows by De Giorgi's Structure Theorem. 
\end{proof}

\section{Proof of Theorem \ref{main}}\label{FE}

As we have established in the past section, we will assume that
for every $K\subset \subset \Omega$ there exist positive constants $\lambda_K,C_K$ such that 
$D(w;x,r) \le C_K$ and
\begin{equation*}\tag{LB}\Per(A,B_r(x)) \ge \lambda_K r^{N-1}, \qquad \forall \; x \in (\partial A \cap K), \forall \;r \in (0,\dist(K,\partial \Omega)).\end{equation*}
{\bf Half-space regularity.} Throughout this section we shall work with the additional assumption for solutions of the half-space problem: let $H \coloneqq \{\;x \in \R^N \; : \; x_N > 0\;\}$ and let $\sigma_H$ be the two-point valued tensor defined in \eqref{eq:definition} for $\Omega = B_1$ (so that $\sigma_ H = \sigma_1$ in $H \cap B_1$), then the operator
\[
P_H u \coloneqq \A^* (\sigma_H \A u) 
\]
is hypoelliptic in $B_1 \setminus \partial H$ in the sense that, if $w \in \Lrm^2(B_1;\R^d)$, then
\begin{equation}\label{eq:hypo}
P_H w = 0 \quad \Rightarrow \quad w \in \Crm^\infty(\overline{B_r^+};\R^d) \cup \Crm^\infty(\overline{B_r^-};\R^d)  \quad \text{for every $0 < r < 1$}. \footnote{The notation $B_r^\pm$ stands for the upper and lower half ball of radius $r$: $B_r \cap H$ and $B_r \cap - H$ respectively.}
\end{equation}
Furthermore, there exists a positive constant $c^* = c^*(N,M,\A)$ such that
\begin{equation}\label{eq:halfplane}
\begin{split}
\frac{1}{\rho^N}\int_{B_\rho} |\nabla^k w|^2 \dd x \le c^*\int_{B_1} |\nabla^k w|^2 \dd x  
\qquad \text{for all $0 < \rho \le \frac{1}{2}$},\\
\frac{1}{\rho^N}\int_{B_\rho} |\A w|^2 \dd x \le c^*\int_{B_1} |\A w|^2 \dd x  
\qquad \text{for all $0 < \rho \le \frac{1}{2}$},\\
\sup_{B_\rho^+ \cup B^-_\rho} |\nabla^{k+1} w|^2 \le 
c^*\int_{B_1} |w|^2 \dd x  
\qquad \text{for all $0 < \rho \le \frac{1}{2}$}.
\end{split}
\end{equation}

\begin{remark}[half-space regularity in applications]\label{nomas} For $1$-st order operators of gradient form it is relatively simple to show that such estimates as in \eqref{eq:halfplane} hold. This case includes gradients and symmetrized gradients; while the linear plate equations may be also reduced to this case (cf. Remark \ref{rem:plate}).
	
	 A sketch of the proof is as follows: The first step is to observe that the tangential derivatives ($i \neq N$) $\partial_i w$ of a solution $w$ of $P_H u = 0$ are also solutions of $P_H u = 0$. The second step is to repeat recursively the previous step and use the Caccioppoli inequality from Lemma \ref{rem:cac} to estimate 
	\begin{equation}\label{eq:caciohalf}
	\int_{B_{1/2}} |\partial^\alpha w|^2 \dd x \le C(|\alpha|)\int_{B_1} |w|^2 \dd x \qquad \text{for arbitrary $\alpha$ with $\alpha_N \le 1$}.
	\end{equation}
	The third step consists in using the ellipticity of $A_N = \mathbb A(\mathbf e_N)$\footnote{Recall that, for a $1$-st order operator as in \eqref{def:ogf}, the coefficients $A_\alpha$ can be simply denoted by $A_i$ with $i = 1,\dots,N$.} (cf. Remark \ref{rem:wave}) and the equation to express $\partial_{NN} w$ in terms of the rest of derivatives: The tensor $(A^T_N \,\sigma \, A_N)$ is invertible, this can be seen from the inequality $|\mathbb A(\mathbf e_N) z|^2 \ge \lambda(\A)|z|^2$ for every $z \in \R^d$ (cf. \ref{rem:wave}) and the fact that $\sigma_H$ satisfies G{\aa}rding's strong inequality \eqref{eq:ellipticity} with $M^{-1}$. Hence, using that $P_H w = 0$, we may write 
\begin{equation}\label{heynckes}
\partial_{NN} w = - (A^T_N \,\sigma_H \, A_N)^{-1} \sum_{ij \neq NN} (A^T_i \, \sigma_1 \, A_j) \partial_{ij} w \quad \text{in $B^+_1$},
\end{equation}
from which estimates for $\partial_{NN} w$ of the form \eqref{eq:caciohalf} in the upper half ball easily follow (similarly for the lower half ball). Further $\partial_N$ differentiation of the equation in $B^\pm_1$ and iteration of this procedure  together with the Sobolev embedding yield bounds as in 
\eqref{eq:halfplane}.

For arbitrary higher-order gradients and other general elliptic systems one cannot rely on the same method. However, the Schauder and $\Lrm^p$ boundary regularity of such systems has been systematically developed in \cite{Agmon1,Agmon2} through the so called {\it complementing condition}. In the case of strongly elliptic systems  (cf. \eqref{eq:ellipticity} and  \eqref{eq:strongly elliptic systems}) this complementing condition is fulfilled, see \cite[pp 43-44]{Agmon2}; see also \cite{Simon97} where a closely related {\it natural notion}  of hypoellipticity of the half-space problem is assumed. 
\end{remark}

\noindent {\bf Flatness excess.} Given a set $A \subset \R^N$ of locally finite perimeter, the {\it flatness excess} of $A$ at $x$ for scale $r$ and with respect to the direction $\nu \in \mathbb S^{n-1}$, is defined as
\begin{align*}{e}(A;x,r,\nu) \coloneqq  \frac{1}{r^{N-1}} \int_{C(x,r,\nu) \cap \partial^*A} \frac{|\nu_E(y) - \nu|^2}{2} \, \dx\mathcal H^{n-1}(y). 
\end{align*}
Here, $C(x,r,\nu)$ denotes for the cylinder centered at $x$ with height $2$, that is parallel to $\nu$, of radius $r$.

Intuitively, the flatness excess expresses for a set $A$, the deviation from being a hyperplane $H$ at a given scale $r$. 
 Again, up to re-scaling, translating and rotating, it will be enough to work the case
$
x = 0,  \nu = \mathbf e_N$, and $r = 1$. 
In this case, we will simply write $\mathrm e(A)$. The hyper-plane energy excess is defined as
\[
 H_{\text{ex}}(w,A;x,r,\nu) \coloneqq  e(A;x,r,\nu) + D(w,A;x,r),
\]
and as long as its dependencies are understood we will simply write $H_{\text{ex}}(r) = e(r) + D(r)$.
  
The following result relies
on the (\ref{LB}) property, a proof can be found in \cite[\S 5.3]{FedererBook} or \cite[Theorem 22.8]{MaggiBook}.

\begin{lemma}[Height bound]\label{HBT} There exist positive constants  $c_1^* = c_1^*(N)$  and $\varepsilon^*_1 =\varepsilon^*_1(N)$ with the following property.
 If $A \subset \R^N$ is a set of locally finite perimeter with the (\ref{LB}) property,
 \[
  0 \in \partial A \quad \text{and}\quad \mathrm e(9) \le \varepsilon_1^*,
 \]
then
\begin{equation*}\label{HB}
 \tag{HB} \sup\bigg\{|y_N| : y \in B'_1 \times [-1,1] \cap \partial A \bigg\} \le c_1^*\cdot  e(4)^\frac{1}{2N -2}.
\end{equation*}
\end{lemma}
 

The next decay lemma is the half-space problem analog of Lemma \ref{lowperi}.
The proof is similar except that it relies 
on the half-space regularity assumptions \eqref{eq:hypo}-\eqref{eq:halfplane} (instead of the ones given by Lemma \ref{lem:cosntant}), and the Height bound Lemma stated above.
\begin{lemma}[approximative solutions of the half-space problem]\label{halfplanedecay} Let $(w,A)$ be a solution of problem \eqref{P} in $B_1$. 
Then, for every $\theta_1^* \in (0,1/2)$ there exist positive constants $c^*_2(\theta_1^*,N,M)$ and $\varepsilon_2^*(\theta^*_1,N,M)$ 
  such that either
  \[
  \int_{B_\rho} |\mathcal A w|^2 \dd x \le c^*_2 \rho^N\|f\|_{\Lrm^\infty(B_1)}^2,
  \]
  or 
  \[
  \int_{B_\rho} |\mathcal A w|^2 \dd x \le 2c^* \rho^{N}\int_{B_1} |\mathcal A w|^2 \dd x 
  \quad \text{for every $\rho \in
  	[\theta_1,1)$},
  \]
   where $c^* = c^*(N,M)$ is the constant from the regularity condition \eqref{eq:halfplane};
  whenever 
  \[
  \Per(A;B_1) \le \varepsilon_2^*.\]
\end{lemma}
\begin{remark}\label{improved}
Let $\delta \in (0,1)$.
Then there exists $\kappa^* = \kappa^*(N,M,\delta)$ such that if $e(1) \le \kappa^*$, and if one further assumes that the excess function $r \mapsto e(r)$ is monotone increasing, then the scaling $w(ry)/r^{(k - \frac{1}{2})}$ and 
the Iteration Lemma~\ref{iteration} imply that
\[
\frac{1}{r^{N-\delta}}\int_{B_r} |\A w|^2 \le C_\delta \big(\|\A w\|^2_{\Lrm^2(B_1)} + c_2^*\| f \|_{\Lrm^\infty(B_1)}^2 \cdot r^{2k + \delta}\big)\qquad \text{for every  $r \in (0,1/2)$},
\]
for some positive constant $C_\delta = C_\delta(N,M)$.
\end{remark}
The next crucial result can be found in \cite[Section 5]{Lin93}. We have decided not to include a proof because because the ideas remain the same: the estimate \eqref{lower}, 
the Height bound Lemma, the Lipschitz approximation Theorem, the estimates from Lemma \ref{halfplanedecay} and the higher integrability for solutions to elliptic equations.\footnote{$\Lrm^{2^*}(\Omega)$-integrability of $\A w$, for some exponent $2^* > 2$, can be established by standard methods through the use of the Caccioppoli inequality in Lemma \ref{rem:cac}.}

\begin{lemma}[flatness excess improvement] \label{key2}  Let $(w,A)$ be a saddle point of problem \eqref{P} in $\Omega$. There exist positive constants $\eta \in (0,1]$, $c^*_3$, and $\varepsilon_3$ depending only on $K$, the dimension $N$, the constant $M$ in \eqref{eq:ellipticity}, and $\|f\|_{\Lrm^\infty}$
with the following properties: If $(w,A)$ is a saddle point of problem \eqref{P} in $B_9$, and
\[
 \hex(9) \le \varepsilon_3^*,
\]
then, for every $r \in (0,9)$, there exists a 
direction $\nu(r) \in \mathbb S^{N-1}$
for which 
$$|\nu(r) - \mathbf e_N| \le c^*_3\,\hex(9) \quad \text{and} \quad \hex(r,\nu(r)) \le c^*_3r^\eta \hex(9).$$
\end{lemma}
\begingroup
\def\thetheorem{\ref{main}}
\begin{theorem}[partial regularity]\label{regularity interface} Let $(w,A)$ be a saddle point of problem \eqref{P} in $\Omega$. Assume that the operator $P_Hu = \A^*(\sigma \A u)$ is hypoelliptic and regularizing as in \eqref{eq:hypo}-\eqref{eq:halfplane}, and that the higher integrability condition 
\[
[\A \tilde u]^2_{\Lrm^{2,N-\delta}(B_{1/2})} \le c\|\A \tilde u\|^2_{\Lrm^2(B_1)}, \quad \text{for some $\delta \in [0,1)$},
\]
holds for every local minimizer $\tilde u$ of the energy $u \mapsto \int_{B_1} Q_{\B} W(\A u)$, where $u \in \Wrm^{\A}(B_1)$. Then there exists a positive constant $\eta \in (0,1]$ depending only on $N$ such that 
\[
 \mathcal H^{N-1}((\partial A \setminus \partial^* A) \cap \Omega) = 0, \quad \text{and} \quad \partial^* A \quad \text{is an open 
 $\Crm^{1,\eta/2}$-hypersurface in $\Omega$}.
\]
Moreover if $\A$ is a first-order differential operator, then $\A w \in \Crm^{0,\eta/8}_{\textnormal{loc}}(\Omega \setminus (\partial A \setminus \partial^* A))$;
and hence, the trace of $\A w$ exists on either side of $\partial^* A$.
\end{theorem}
\addtocounter{theorem}{-1}
\endgroup
\begin{proof}
\noindent{\bf The reduced boundary is an open hypersurface.} 
The first assertion $\mathcal H^{N-1}((\partial A \setminus \partial^* A) \cap \Omega) = 0$ is a direct consequence of 
Corollary \ref{essential}.

To see that $\partial^* A$ is relatively open in $\partial A$ we argue as follows: De Giorgi's Structure Theorem guarantees that for every $x \in \partial^*A$ there exist $r > 0$  (sufficiently small) and  $\nu \in \mathbb S^{N-1}$ such that 
\[
\hex(w,A;r,x,\nu) \le \frac{1}{2} \varepsilon^*_3, \quad \text{and} \quad \mu_A(\partial B_r(x)) = 0.
\]
The map $y \mapsto \mu_A(B_r(y)) = 0$ is continuous at $x$, therefore we may find $\delta(x) \in (0,1)$ such that
\begin{equation*}
\hex(w,A;r,y,\nu) \le \varepsilon^*_3 \quad \text{for every $y \in B_\delta(x) \cap \partial A$}.
\end{equation*}
We may then apply Lemma \ref{key2}  to get an estimate of the form
\[
\inf_{\xi \in \mathbb S^{N-1}} \hex(w,A;y,\rho,\xi) \le c^*_3\rho^\eta \hex(w,A;y,r,\nu) \quad \text{for all $y \in B_\delta(x)$, and all $\rho \in (0,r)$}.
\]
This and the first assertion of Lemma \ref{key2} imply that $y \in \partial^* A$ for every $y \in B_\delta(x) \cap \partial A$. Therefore, the reduced boundary $\partial^* A$ is a relatively open subset of the topological boundary $\partial A$.

We proceed to prove the regularity for $\partial^* A$.
It follows from the last equation that 
\begin{equation}\label{final}
D(w;y,\rho) \le \inf_{\xi \in \mathbb S^{N-1}} \hex(w,A;y,\rho,\xi) \le c^*_3\varepsilon^*_3\rho^\eta \le C\rho^{\eta}\end{equation}
\text{for every $y \in B_\delta(x)$, and every $\rho \in (0,r)$},
for some constant $C = C(C_{B_\delta(x)},\Lambda_{B_\delta(x)},N,M)$.

Through a simple comparison, we observe from \eqref{final} and the property that $(w,A)$ is a local saddle point of problem \eqref{P} in $B_\delta(x)$, that
\[
\Dev_{B_\delta(x)}(A,\rho) \le 2M \rho^{N-1}D(w;y,\rho) \le 2MC\rho^{N-1 + \eta}, \quad \text{for all $\rho \in (0,r)$ and every $y \in B_\delta(x)$}.
\]
We conclude with an application of 
Tamanini's Theorem \ref{tamanini}: 
\[\partial A = \partial^* A \; \text{ is a $\Crm^{1,\eta/2}$-hypersurface in 
$B_\delta(x)$}.\] 
The assertion follows by observing that the regularity of $\partial^* A$ is a local property.

\noindent{\bf Jump conditions for the hyper-space problem.} Let $\tau \in \Lrm_{\text{loc}}^2(B_1;Z) 
\cap (\Crm^\infty(\overline{B_\rho^+};Z)
\cup \Crm^\infty(\overline{B_\rho^-};Z))$ for every $\rho \in (0,1)$, assume furthermore that $\tau$ is a solution of the equation
\[
\A^* \tau = 0 \qquad \text{in $B_1$}. 
\]
Let $\eta \in \Crm^{\infty}_c(B_1';\R^d)$ be an arbitrary test function and choose a function 
$\varphi \in \Crm^{\infty}_c(B_1;\R^d)$ with the following property:
\[\varphi(y',y_N) = \frac{y_N^{k-1}}{{(k-1)}!} \eta(y') \qquad \text{in a neighborhood of $B_1'$}.\] 
Then, integration by parts and Green's Theorem yield that
\[
0 = \int_{B_1}  \tau \cdot \A \varphi \dd y = \int_{\partial H \cap B_1} [\mathbb A(\mathbf e_N)^T \cdot \tau] \cdot \eta \dd y',
\]
where $[\mathbb A(\mathbf e_N)^T \cdot \tau] = \mathbb A(\mathbf e_N)^T \cdot (\tau^+ - \tau^-)$. Here, 
$\tau^+$ and $\tau^-$ are the traces of $\tau$ in $\partial H$ from $B_1^+$ and $B_1^-$ respectively. 
Since $\eta$ is arbitrary, a density argument shows that
\begin{equation}\label{eq:jump}
[\mathbb A(\mathbf e_N)^T \cdot \tau] = 0 \quad \text{in $\partial H \cap B_1$, \qquad and hence \quad $\mathbb A(\mathbf e_N)^T \cdot \tau \; \in \Wrm_{\text{loc}}^{1,2}(B_1;\R^d)$}.
\end{equation}

\noindent{\bf Regularity of $\A w$.} From this point  and until the end of the proof we further assume that $\A$ is a first-order differential operator of gradient form; we may as well assume that $\partial^* A$ is locally parametrized by $\Crm^{1,\eta/2}$ functions. 

Due to Campanato's Theorem ($\Crm^{0,\eta/8} \simeq \mathrm L^{2,\,N + (\eta/4)}$ on Lipschitz domains), our goal is to show local boundedness of the map
\begin{equation}\label{eq:prima}
x \mapsto \sup_{r \le 1} \bigg\{\frac{1}{r^{N + (\eta/4)}} \int_{B_r(x) \cap A} |\A w - (\A w)_{B_r(x) \cap A}|^2 \dd y\bigg\} \qquad x \in (\Omega \setminus  (\partial A \setminus \partial^* A));
\end{equation} 
and a similar result for $A^c$ instead of $A$.

Also, since Campanato estimates in the interior are a simple consequence of 
Lemma \ref{lem:cosntant}, we may restrict our analysis to show only local boundedness at points $x \in \partial^* A$. We first prove the following decay for solutions of the half-space:

\begin{lemma}\label{lem:R} Let $\tilde w \in \Wrm^{\A}(B_1)$ be such that 
\begin{equation}\label{eq:claim}
\A^*(\sigma_H \A \tilde w) = 0 \quad \text{in $B_1$}.
\end{equation}
Then $\tilde w$ satisfies an estimate of the form
\begin{equation}\label{Calpha}
\frac{1}{\rho^{N + 2}} \int_{B_\rho} |R_H \tilde w - (R_H \tilde w)_\rho|^2 \dd y \le c(N,\sigma_1,\sigma_2) \int_{B_1} |R_H \tilde w - (R_H \tilde w)_1|^2 \dd y \qquad \text{for all $0 < \rho \le 1$},
\end{equation}
where we have defined 
\[
R_A u \coloneqq \big(\nabla'u, A_N^T( \sigma_A \A u)\big), \qquad A \subset B_1 \; \text{Borel}.
\]
\end{lemma}
\begin{proof} Since for $\rho \ge 1/2$ one can use $c \coloneqq 2^{-{(N+2)}}$, we only focus on proving the estimate for $\rho \in (0,1/2)$. It is easy to verify that
$\A^* (\sigma_H \A (\partial_i \tilde w - \lambda)) = 0$ in $\mathcal D'(B_1;\R^d)$ for all $\lambda \in \R^d$, and every $i = 1,\dots,{N-1}$. 
In particular, by \eqref{eq:halfplane} we know that
\begin{equation}\label{eq:u'}
	\frac{1}{\rho^{N + 2}}\int_{B_\rho} |\partial_i \tilde w - (\partial_i \tilde w)_{\rho}|^2 \dd y \le \frac{C}{\rho^N}\int_{B_\rho} |\nabla \partial_i \tilde w|^2 \dd y \le c^*\, C \int_{B_1} |\partial_i \tilde w - (\partial_i \tilde w)_1|^2 \dd y,
\end{equation}
\text{for every $\rho \in (0,1/2)$},
and every $i = 1,\dots,N - 1$. Here, $C = C(N)$ is the standard scaled Poincar\'e constant for balls.
Summation over $i \in \{1,\dots,N-1\}$ yields an estimate of the form \eqref{Calpha} for $\nabla' \tilde w$. 

We are left to calculate the decay estimate for $g_H(\tilde w) \coloneqq A_N^T (\sigma_H \A \tilde w) = \mathbb A(\mathbf e_N) \cdot (\sigma_H \A \tilde w)$. By the hypoellipticity assumption \eqref{eq:hypo} and the jump condition \eqref{eq:jump}, we infer that 
$g_H(\tilde w) \in \Wrm_{\text{loc}}^{1,2}(B_1;\R^d)$. 

Even more, by the same Poincar\'e's inequality 
\begin{equation}\label{eq:please}
\frac{1}{\rho^{N +2}}\int_{B_\rho} |g(\tilde w) - (g(\tilde w))_\rho|^2 \dd y \le 
\frac{C}{\rho^{N}}\int_{B_\rho \setminus \partial H} |\nabla (g(\tilde w)) |^2 \dd y 
\end{equation}
for every $\rho \in (0,1/2)$.
On the other hand, it follows from the equation in $(B_1 \setminus \partial H)$ and  \eqref{heynckes} that one may write $\nabla g(\tilde w)$ in terms of $\nabla (\nabla' \tilde w)$ for almost every $x \in (B_r \setminus \partial H)$. We may then find a constant $C' = C'(\sigma_1,\sigma_2,\A)$ such that
\begin{equation*}
 |\nabla g(\tilde w(x))|^2\le C' |\nabla (\nabla' \tilde w)(x)|^2 \qquad \text{for every $x \in (B_\rho \setminus \partial H)$}.
\end{equation*}
Using the same calculation as in the derivation of \eqref{eq:u'}, it follows from \eqref{eq:please} that
\begin{align*}
\frac{1}{\rho^{N +2}}\int_{B_\rho} |g(\tilde w) - (g(\tilde w))_\rho|^2 \dd y & 
\le c^*\, C\, C' \int_{B_1} |\nabla' \tilde w -  (\nabla' \tilde w)_1|^2 \dd y \\
&
\le c^*\, C\, C' \int_{B_1} |R_H\tilde w - (R_H\tilde w)_1|^2 \dd y,
\end{align*}
for every $\rho \in (0,1/2)$.
The assertion follows by letting $c(N,\sigma_1,\sigma_2) \coloneqq c^*\, C\max\{1,C'\}$.
\end{proof}
The next corollary can be inferred from \eqref{Calpha} by following the strategy of Lin in \cite[pp 166-167]{Lin93}:

\begin{corollary}\label{cor:R}
Let $\tilde w \in \mathrm W^{\A}(B_2)$ solve the equation
\begin{equation}\label{porfavor}
\A^*(\sigma_A \A u) = f \quad \text{in $B_2$}, \qquad \text{with \quad $\|\tilde w \|_{\Lrm^2(B_2)} \le 1$ and $\| f \|_{\Lrm^\infty(B_2)} \le 1$ },
\end{equation}
where $A \coloneqq \{\; x \in B'_2 \times \R \; : \; x_N > \varphi(x')\; \}$ for some function $\varphi \in \Crm^{1,\eta/2}(B_2')$ with $\varphi(0) = |\nabla \varphi|(0) = 0$, and $\| \varphi \|_{\Crm^{1,\eta/2}(B_2')} \le 1$.
Then there exist positive constants $\theta(N,\sigma_1,\sigma_2) \in (0,1/2)$, and $C(N,\sigma_1,\sigma_2)$ such that either 
\begin{equation}\label{eq:b1}
\frac{1}{\theta^{N +1}}\int_{B_\theta} |R_A \tilde w - (R_A \tilde w)_\theta|^2 \dd y \le \int_{B_{1}} 
	 |R_A \tilde w - (R_A \tilde w)_{1}|^2 \dd y,
\end{equation}
or
\begin{equation}\label{eq:b2}
\int_{B_\theta} |R_A \tilde w - (R_A \tilde w)_\theta|^2 \dd y \le C\bigg(\|\varphi \|_{\Crm^{1,\eta/2}(B_1')} + \|f\|^2_{\Lrm^\infty(B_1)} \bigg).
\end{equation} 
\end{corollary}

We are now in the position to prove \eqref{eq:prima}. Let $\delta \in (0,\eta/2)$ and let $(w,A)$ be solution of problem $\eqref{P}$. Since local regularity properties of the pair $(w,A)$ are inherited to any (possibly rotated and translated) re-scaled pair $(w^{x,r},A^{x,r})$ -- as defined in \eqref{blow2}, where in particular the source $f^{x,r}$ tends to zero -- with $r \le \dist(x,\partial \Omega)$,
we may do the following assumptions without any loss of generality: $B_4 \subset \Omega$ and $x = 0 \in \partial^* A$, $\partial A^*$ is parametrized in $B_{2}$ by a function $\varphi \in \Crm^{1,\eta/2}(B_{2}')$ such that $\varphi(0) = |\nabla \varphi(0)| = 0$, and $\|\varphi\|_{\Crm^{1,\eta/2}(B_{2}')}, \|f\|_{\Lrm^\infty(B_2;\R^d)} \le \min\{1,\kappa^*\}$ where $\kappa^* = \kappa^*(\delta,N,M)$ is the constant of Remark \ref{improved}.
Additionally, since $(w,A)$ is a solution of problem \eqref{P}, we know that
\begin{equation}\label{joder}
\A^* (\sigma_{A} \A w) = f \qquad \text{in $B_2$},\end{equation}
and
\begin{equation}\label{memo}
\frac{1}{r^{N - \delta}} \int_{B_r} |\A w|^2 \dd y \le C_\delta\big(\|\A w\|^2_{\Lrm^2(B_2)} + \| f \|^2_{\Lrm^\infty(B_1)}\big) \qquad \text{for every $r \in (0,1)$},
\end{equation}
where $C_\delta(N,M)$ is the constant from 
Remark $\ref{improved}$.

Notice that the rescaled functions\footnote{Here, $\nu_r$ is the $\A$-free corrector function for $w$ in $B_r$, see Definition \ref{gto}.} $w^r(y) \coloneqq (w(ry) - v_r(ry))/r^{1 - (\delta/2)}$ and $\varphi^r(y) \coloneqq \varphi(ry)/r$ 
still solve \eqref{joder}  for $f^r(y) \coloneqq r^{1 + (\delta/2)} f(ry)$ and $A^r \coloneqq A/r$ with $\|\varphi^r\|_{\Crm^{1,\eta/2}(B_{2}')}, \newline\|f^r\|_{\Lrm^\infty(B_2;\R^d)} \le \min\{1,\kappa^*\}$. In particular, by \eqref{memo} and Poincar\'e's inequality
\[
 \|w^r\|^2_{\Lrm^2(B_1)} \le C(B_1)\|\A w^r\|^2_{\Lrm^2(B_1)} < \overline C \coloneqq C(B_1)C_\delta\big(\|\A w\|^2_{\Lrm^2(B_2)} + 1\big).
\] 
Thus
Recall also that $\|\varphi^r\|_{\Crm^{1,\eta/2}(B_1')}$ scales as $r^{\eta/2}\|\varphi\|_{\Crm^{1,\eta/2}(B_r')}$ and, in view of its definition, $\| f^r\|^2_{\Lrm^\infty(B_1)}$ scales as $r^{2 + \delta}$.
In view of these properties, we are in position to apply {Corollary \ref{cor:R}} to $w^r/\max\{1,\overline C^{1/2}\}$: We infer that either
\begin{equation}\label{eq:R1}
\frac{1}{\theta^{N +1}}\int_{B_{\theta}} |R_{A^r}   w^r - (R_{A^r}   w^r)_{\theta}|^2 \dd y \le \int_{B_{1}} 
	 |R_{A^r}   w^r- (R_{A^r}  w^r)_{1}|^2 \dd y,
\end{equation}
or
\begin{equation}\label{eq:R2}
\int_{B_{\theta}} |R_{A^r} w^r - (R_{A^r} w^r)_{\theta}|^2 \dd y \le \max\{1,
\overline C\}\cdot C(N,\sigma_1,\sigma_2)\bigg(\|\varphi^r\|_{\Crm^{1,\eta/2}(B_1')} + r^{2 + \delta}\bigg),
\end{equation}
where $\theta = \theta(N,\sigma_1,\sigma_2) \in (0,1/2)$ is the constant from Corollary \ref{cor:R}.

 It is not difficult to verify, with the aid of the Iteration Lemma \ref{iteration}, that re-scaling in \eqref{eq:R1} and \eqref{eq:R2} conveys a decay of the form
\begin{equation}\label{joder22}
\frac{1}{r^{N + \eta/2 - \delta}}\int_{B_{r}} |R_A  (w - \nu_r) - (R_A   (w - \nu_r))_{r}|^2 \dd y  \le
c' \qquad \text{for all $r \in (0,1)$},
\end{equation}
and some constant $c' = c'(\delta,N, \sigma_1,\sigma_2,\|\A w\|_{\Lrm^2(B_2)}$).

The last step of the proof consists in showing that $R_A  (w - \nu_r)$ dominates $\nabla  (w - \nu_r)$. By the definition of $R_A$, it is clear that $|\nabla'  (w - \nu_r)(x) - (\nabla'  (w - \nu_r))_{B_r \cap A}|^2 \le |R_A  (w - \nu_r)(x) - (R_A  (w - \nu_r))_{B_r \cap A}|^2$ for all $x \in B_1$ and every $r \in (0,1)$. We show a similar estimate for $\partial_N  (w - \nu_r)$:

 The pointwise G{\aa}rding inequality \eqref{eq:ellipticity} and \eqref{eq:strongly elliptic systems} imply, in particular, that the tensor $(\mathbb A(\mathbf e_N)^T \sigma_1 \; \mathbb A(\mathbf e_N)) = (A_N^T \sigma_1 A_N) \in \text{Lin}(\R^d;\R^d)$ is invertible  (use, e.g., Lax-Milgram in $\R^d$). Hence,
\begin{align}\label{rhs}
\partial_N  (w - \nu_r) = & (A_N^T \sigma_1 A_N)^{-1}\bigg(
g(w - \nu_r) \; -
\sum_{j \neq N} (A_N^T \sigma_1 \; A_j) \partial_j  (w - \nu_r) 
\bigg) \qquad \text{in $B_1 \cap A$},
\end{align}
from where we deduce that
\begin{align*}
\frac{1}{r^{N + (\eta/2) - \delta}}\int_{B_r \cap A} |\partial_N   &(w - \nu_r) - (\partial_N  (w - \nu_r))_{B_r \cap A} |^2 \dd y \le \\ & \frac{c''}{r^{N + (\eta/2) - \delta}} \int_{B_r \cap A} |R_A  (w - \nu_r) - (R_A  (w - \nu_r))_{B_r \cap A}|^2 \dd y
\end{align*}
for some constant $c'' = c''(\sigma_1) \ge 1$ bounding the right hand side of \eqref{rhs} in terms of $\nabla' w$ and $g(w)$.

By \eqref{joder2} and the estimate above we obtain
\begin{align*}
\frac{1}{r^{N + (\eta/2) - \delta}}\int_{B_r \cap A} & |\A  w - (\A  (w))_{B_r \cap A} |^2 \dd y = \\ &
\frac{1}{r^{N + (\eta/2) - \delta}}\int_{B_r \cap A} |\A  (w - \nu_r) - (\A  (w - \nu_r))_{B_r \cap A} |^2 \dd y \\
& \le \frac{C(\A)}{r^{N + (\eta/2) - \delta}} \int_{B_r \cap A} |\nabla  (w - \nu_r) - (\nabla  (w - \nu_r))_{B_r \cap A} |^2 \dd y \\
& \le \overline c(N,\sigma_1,\sigma_2,\|\A w\|_{\Lrm^2(B_2)}) \coloneqq C(\A) \cdot c' \cdot c'',
\end{align*}
for every $r \in (0,1)$. The assertion follows by taking $\delta = \eta/4$.

Notice that the dependence on $\| \A w \|_{\Lrm^2(B_2)}$ is local since we assumed $B_4 \subset \Omega$; this means that in general we may not expect a uniform boundedness of the decay. Similar bounds for $A$ replaced by $A^c$ can be derived by the same method. 

\end{proof}
\begin{remark}[regularity I] In general, for a $k$-th order operator $\A$ of gradient form, the only feature required to prove the regularity of $\nabla^k w$ up to the boundary $\partial^*A$ by the same methods as for first-order operators of gradient form is to obtain an analog of Lemma \ref{lem:R} (and its Corollary \ref{cor:R}) for higher-order operators. 

More specifically, if $\tilde w \in \Wrm^{\A}(B_1)$ is a solution of the equation 
\[
\A^*(\sigma_H \A u) = 0 \quad \text{in $B_1$},
\]
then $\tilde w$ satisfies an estimate of the form
\begin{equation}\label{Calpha2}
\frac{1}{\rho^{N + 2}} \int_{B_\rho} |R_H \tilde w - (R_H \tilde w)_\rho|^2 \dd y \le c(N,\sigma_1,\sigma_2) \int_{B_1} |R_H \tilde w - (R_H \tilde w)_1|^2 \dd y \qquad \text{for all $0 < \rho \le 1$},
\end{equation}
where
\[
R_A u \coloneqq \big(\nabla'u, \mathbb A(\mathbf e_N)^T (\sigma_A \A u)\big), \qquad A \subset B_1.
\]

Unfortunately, for $2k$-order systems of elliptic equations (with $k > 1$) it is not clear to us whether one can prove such decay estimates by standard methods. While a decay estimate for $\nabla^{k-1}(\nabla' u)$ can be shown by the very same method as the one in the proof of Theorem \ref{main}, the main problem centers in proving a decay estimate for the term $\mathbb A(\mathbf e_N)^T (\sigma \A u) \in \Wrm^{1,2}(B_1)$ -- cf. \eqref{eq:jump}. Technically, the issue is that one cannot use the equation on half-balls to describe $\partial^{(0,\dots,0,k)} u$ in terms of $\nabla^{k-1}(\nabla' u)$. 
\end{remark}

\begin{remark}[regularity II - linear plate theory]
In the particular case of models in linear plate theory ($\A = \nabla^2, N= 2$, and $d = 1$) it is possible to show a decay estimate as in \eqref{Calpha2} for solutions $w \in \Wrm^{2,2}_0(B_2)$ of the equation
\[
\nabla \cdot \nabla (\sigma_H \nabla^2 u) = 0.
\]

By Remark \ref{rem:plate},
there exists a field $w \in W^{1,2}(B_2;\R^2)$ which turns out to be a solution of the equation 
\[
\nabla \cdot (\mathbf S_H	\, \mathcal E w) = 0,
\]
where $\mathbf S$ is a positive fourth-order symmetric tensor such that $\sigma_H(x) = \mathbf R_\perp \mathbf S_H^{-1}(x) \mathbf R_\perp$; furthermore, $\mathbf R_\perp \mathcal E w =\mathbf \sigma_H \nabla^2 u$. 
Since $\A = \nabla^2$, it is easy to verify that $A_\alpha= A_{(i,j)} = \mathbf e_i \otimes \mathbf e_j$ for $i,j \in \{1,2\}$, a simple calculation shows that
\[
g_H(u) \coloneqq \mathbb A (\mathbf{e}_N)^T  (\sigma_H\A u) = (\sigma_H \nabla^2 u)_{22} = 
(\mathbf R_\perp \mathcal E w)_{22} = \partial_1w^1;
\] 
and thus, since $\mathcal E$ is an operator of gradient form of order one, it follows form the proof of Theorem \ref{main} that an estimate of the form \eqref{Calpha2} indeed holds for $g_H(u)$.
\end{remark}

\subsection*{Acknowledgments} 
 I wish to extend many thanks to Prof. Stefan M\"uller for his advice and fruitful discussions in this beautiful subject. I would also like to thank the reviewer and the editor for their patience and care which  derived in the correct formulation of Theorem \ref{main}.
The support of the University of Bonn and the Hausdorff Institute for Mathematics is gratefully acknowledged. The research conducted in this paper forms part of the author's Ph.D. thesis at the University of Bonn.

\section{Glossary of constants}

\begin{tabular}{l l}
	$N   \qquad$  & spatial dimension \\
	$M$  & coercivity and bounding constant for the tensors $\sigma_1$ and $\sigma_2$ (as quadratic forms) \\
	$K$  & an arbitrary compact set in $\Omega$ \\
	$\lambda_K$ & local upper bound constant
\end{tabular}\\

\noindent {\bf Other constants:} groups of constants are numbered in non-increasing order, e.g., $c_1^* \ge c_2^* \ge c_3^*$. The following constants play an important role in our calculations:
\begin{center}
	\begin{tabular}{|c|c|c|} \hline
		Constant & Dependence & Description \\ \hline
		$\theta_1$ & arbitrary in $(0,1/2)$ &  ratio constant	\\
		$c_1$ & $\theta_1,N,M$ &  universal constant \\ 
		$\epsilon_1$ & $\theta_1, N, M, \theta_1$ & smallness of perimeter density \\
		$c_2$ & $M$ & universal constant \\ 
		$\gamma$ & $N,M$ & universal constant \\
		$\theta_2$ & $N,M$ & universal constant \\
		$\varepsilon_2$ &  $N,M$ & smallness of excess energy \\ 
		$c^*_1$ & $\lambda_K, N$ & constant in the Height bound Lemma  \\ 
		$\theta_1^*$ & arbitrary in $(0,1/2)$ &  ratio constant	\\
		$c_2^*$ & $\theta_1^*, N,M$ & universal constant \\ 
		$\epsilon_2^*$ & $\theta^*_1,N, M$ & smallness of flatness excess  \\
		$c_3^*$ & $K,N,M,f$ & flatness excess improvement scaling constant \\
		$\varepsilon_3^*$ & $K,N,M,f$ &  smallness of flatness excess \\  \hline
	\end{tabular}
\end{center}

\bibliographystyle{spmpsci}      
\bibliography{Biblio}   


\end{document}